\documentclass[11pt]{amsart}
\RequirePackage{pdf14}
\RequirePackage{fix-cm}
\usepackage[left=2.5cm,right=2.5cm,top=3.5cm]{geometry}
 \usepackage[utf8]{inputenc}

\usepackage{graphicx}
\usepackage{hyperref}
\usepackage{ytableau}
\usepackage{enumerate}
\usepackage{amsfonts,amsmath,amssymb,amsthm}
\usepackage{tikz,tikz-cd}
\usepackage[all]{xy}
\usetikzlibrary{positioning} 
\usepackage{verbatim}
\usepackage{array,setspace,mathrsfs,yfonts,dsfont,bbm,colonequals,amscd,euscript}
\usepackage{relsize,suffix,mathtools,cancel,bbm}

\makeatletter
\newcommand\mathcircled[1]{%
	\mathpalette\@mathcircled{#1}%
}
\newcommand\@mathcircled[2]{%
	\tikz[baseline=(math.base)] \node[draw,circle,inner sep=3pt] (math) {$\m@th#1#2$};%
}
\makeatother

\numberwithin{equation}{subsection}

\theoremstyle{plain}
\newtheorem{thm}{Theorem}[section]
\newtheorem{lem}[thm]{Lemma}

\newtheorem{prop}[thm]{Proposition}
\newtheorem{cor}[thm]{Corollary}

\newtheorem*{thm*}{Theorem}

\theoremstyle{definition}
\newtheorem{dfn}[thm]{Definition}
\newtheorem{exm}[thm]{Example}

\theoremstyle{remark}
\newtheorem{rem}[thm]{Remark}

\def\IA{\mathbb{A}}

\def\IC{\mathbb{C}}
\def\ID{\mathbb{D}}

\def\IN{\mathbb{N}}

\def\IV{\mathbb{V}}
\def\IZ{\mathbb{Z}}
\def\IG{\mathbb{G}}
\def\IO{\mathbb{O}}

\def\AA{{\mathcal A}} 

\def\CC{{\mathcal C}}
\def\DD{{\mathcal D}}

\def\HH{{\mathcal H}}

\def\JJ{{\mathcal J}}

\def\OO{{\mathcal O}}

\def\WW{{\mathcal W}}

\def\GL{\mathrm{GL}}
\def\PGL{\mathrm{PGL}}

\def\ad{\mathrm{ad}}

\newcommand{\ket}[1]{|{#1}\rangle}

\def\gf{\mathfrak{g}}

\def \nf{\mathfrak{n}}
\def \hf{\mathfrak{h}}
\def \slf{\mathfrak{sl}}
\def \glf{\mathfrak{gl}}
\def \Gf{\mathfrak{G}}

\newcommand{\bb}{\mathbb}

\newcommand{\mf}{\mathfrak}

\def \eg{{\it e.g.}}
\def \ie{{\it i.e.}}

\def \KL{\mathrm{KL}}
\def \id{\mathbbm{1}}
\def \ik{\mathbb{C}}

\newcommand{\Hi}[1]{\mathrm{H}^{\tfrac{\infty}{2}+#1}}

\newcommand{\fibre}[1]{\underset{#1}{\times}}
\newcommand{\tensor}[1]{\underset{#1}{\otimes}}

\newcommand{\im}{{\textup{im}}}

\newcommand{\End}{{\textup{End}}}

\newcommand{\del}{{{\partial}}}

\newcommand{\Ph}{{{[\![\hbar]\!]}}}
\newcommand{\Lz}{{{(\!(z)\!)}}}

\newcommand{\Pt}{{{[\![t]\!]}}}

\newcommand{\red}[1]{/\!\!/\!\!/_{#1}}

\newcommand{\open}[1]{\overset{\circ}{#1}}

\title[Inverse Hamiltonian Reduction for affine W-algebras in type A]{Inverse Hamiltonian Reduction for affine W-algebras in type A}
\author{Dylan Butson and Sujay Nair}
\date{}

\begin{document} 
\begin{abstract}
We give a geometric proof of inverse Hamiltonian reduction for all affine W-algebras in type A at generic level, a certain embedding of the affine W-algebra corresponding to an arbitrary nilpotent in $\mathfrak{gl}_N$ into that corresponding to a larger nilpotent with respect to the closure order on orbits, tensored with an auxiliary algebra of free fields. We proceed by constructing strict chiral quantizations of equivariant Slodowy slices, sheaves of $\hbar$-adic vertex algebras on the arc spaces of the slices, and then localizing them on quasi-Darboux open sets. We also provide a generalization for the Drinfeld--Sokolov reduction of arbitrary vertex algebra objects in the Kazhdan--Lusztig category.
\end{abstract}
\maketitle
\vspace*{-\baselineskip}
\tableofcontents


\section{Introduction}

Given a nilpotent element $\chi$ in $\gf^* = \glf_N^*$ and a level $k\in \IC$, one can construct a certain vertex algebra, $\WW^k(\gf,\chi)$, called the affine W-algebra. In fact, up to isomorphism this vertex algebra only depends on the nilpotent orbit of $\chi$, and so can be labeled by $\mu= [\mu_1\ge \mu_2\ge \mu_3\ge\dots \mu_{n-1}\geq 0]$ a partition of $N$, and we identify the set of all such partitions of length $n-1$ with the set of dominant, integral coweights $\Lambda^+(\PGL_n)$ of $\PGL_n$.

For the special case where $\mu=[1^N]$, \ie, $\chi=0$, this vertex algebra is the universal affine vertex algebra $V^k(\gf)$. One has a homological construction of W-algebras for $\mu\ge [1^N]$ by quantum Hamiltonian (or Drinfeld--Sokolov) reduction of $V^k(\gf)$, see \cite{Kac:2003DS}. This construction has proven to be of great value in studying the representation theory of these algebras \cite{Arakawa:2007invent}.

Recently, for $\mu_1\ge\mu_2$ satisfying certain hypotheses, it has been shown by \cite{Genra:2025redbystages} that one can obtain $\WW^k_{\mu_1}$ from $\WW^k_{\mu_2}$ by a kind of partial Hamiltonian reduction. The main result of this work concerns \textit{inverse Hamiltonian reduction}, which realizes $\WW^k_{\mu_2}$ as a subalgebra of the tensor product of $\WW^k_{\mu_1}$ with so-called free fields.

\begin{thm}\label{THM:main W algebra}
	Let $\mu\in \Lambda^+(\PGL_n)$ and $\alpha$ a positive co-root such that $\mu+\alpha\in \Lambda^+(\PGL_n)$. Then, at generic level $k$, we have an embedding of $\hbar$-adic vertex algebras 
	\begin{equation}
		\WW_{\hbar,\mu}^k\hookrightarrow \WW^k_{\hbar,\mu+\alpha}\hat{\otimes} \DD^{ch}_{\rm loc}(\IA^{\rm len \,\alpha})~,
	\end{equation}
	for $\DD^{ch}_{\rm loc}(\IA^{\rm len\,\alpha})$ a localization of the $\hbar$-adic vertex algebra of chiral differential operators on $\IA^{\rm len \,\alpha}$.
\end{thm}

This theorem completes the generalization to type A of a family of results going back to \cite{Semikhatov:inv}, as we will explain below. It is obtained in much the same way as in \cite{BN:finite}, where we established the analogous result for finite W-algebras. Indeed, crucial ingredients in the proof are the geometric results on equivariant Slodowy slices obtained in \textit{loc.\ cit.}. The other crucial ingredient, like in finite type, is deformation theory---now of vertex algebras as opposed to associative algebras. In a sense, this work is an application of the results of \cite{BN:def} on deformations of quantizations of arc spaces.

Our proof proceeds by first constructing quantizations of the arc spaces of the equivariant Slodowy slices. Namely, for each $k\in \IC$, there is a sheaf of $\hbar$-adic vertex algebras $\AA^k_{\hbar,G,\mu}$ on $\JJ S_{G,\mu}$ whose $\hbar$-adic vertex algebra of global sections
is isomorphic to the $\hbar$-adic analogue of the equivariant affine W-algebra at level $k$ of \cite{Arakawa:2018egx}, that is, 
\[
\Gamma(\JJ S_{G,\mu} , \AA^k_{\hbar,G,\mu}) \cong {\bf W}^k_{\hbar,G,\mu}~.
\] We then prove inverse Hamiltonian reduction for these equivariant W-algebras, constructing the inclusion by localizing this sheaf on the image of the inverse Hamiltonian reduction chart of \cite{BN:finite} and appealing to the deformation theory results of \cite{BN:def}.
\begin{thm}
	Let $\mu\in \Lambda^+(\PGL_n)$ and $\alpha$ a positive co-root such that $\mu+\alpha\in \Lambda^+(\PGL_n)$. Then, at generic level $k$, we have an embedding of $\hbar$-adic vertex algebras 
	\begin{equation}
		{\bf W}_{\hbar,G,\mu}^k\hookrightarrow {\bf W}^k_{\hbar,G,\mu+\alpha}\hat{\otimes} \DD^{ch}_{\rm loc}(\IA^{\rm len \,\alpha})~.
	\end{equation}
\end{thm}
Inverse hamiltonian reduction for the ordinary affine W-algebras follows by taking $G(\OO)$-invariants. The result for equivariant W-algebras in fact implies a stronger version of inverse Hamiltonian reduction for arbitrary modules in the Kazhdan--Lusztig category. Namely, we prove the following:
\begin{thm}
	Let $M\in {\rm KL}_k$, and denote by $H^0_{\rm DS}(M,\mu)$ the Drinfeld--Sokolov reduction of $M$ at some representative of the orbit labeled by $\mu$. Then we have natural embeddings of $\WW^k_{\hbar,\mu}$ modules
	\begin{equation}
	H^0_{\rm DS}(M,\mu) \hookrightarrow H^0_{\rm DS}(M,\mu+\alpha)\hat{\otimes} \DD^{ch}_{\rm loc}(\IA^{\rm len\,\alpha})~,
	\end{equation}
	where the $\WW_{\hbar,\mu}^k$-module structure on the right comes from the embedding in Theorem \ref{THM:main W algebra}. In particular, if $M$ is a vertex algebra object in $\KL_k$, this is an embedding of vertex algebras.
\end{thm}

\subsection{Relation to other work}

The first example of inverse Hamiltonian reduction for affine W-algebras was the work of Semikhatov \cite{Semikhatov:inv}, which considered the $\slf_2$ case. More recently, higher rank and superalgebra examples have been constructed by Adamovi\'c in \cite{Adamovic:2004zi, Adamovic2019ER}. The first somewhat systematic proof of inverse Hamiltonian reduction in type $A$ is due to Fehily \cite{Fehily:2023mcg}, who proved the statement for hook-type W-algebras using the free-field realizations of these W-algebras from \cite{Genra2017scre}. This approach was extended to all examples in rank three in \cite{Creutzig:2025struct,Fasquel:2024ssg}, and further to all Virasoro-type reductions in \cite{Fasquel:virasoro}.

Additionally, the present work also offers a novel proof of what is known in the literature as reduction by stages. This is a result first conjectured, for finite type, in Morgan's thesis \cite{Morgan:2015redbystages} and proven in the finite and affine case by \cite{Genra:2024finred,Genra:2025redbystages}. We note, however, that our hypotheses differ from those of \textit{loc.\ cit.} in Type A.

In a sense, inverse Hamiltonian reduction can be viewed as a special case of a generalized coproduct on truncations of the shifted affine Yangian of $\glf_n$. For the finite Yangian, this comultiplication was described in \cite{Finnkelberg:2018comult}. In the affine case, however, constructing the general coproduct has proven to be a challenging endeavor. In the case of rectangular nilpotents, some results have been obtained in \cite{Kodera:2022coproduct,Gaiotto:2023ynn} and geometric approaches to this problem can be found in \cite{BR1, Butson:2023fcv}.


Unlike \cite{BN:finite}, we make almost no reference to the affine Grassmannian, or generalized slices therein. While the results of our previous work related various quantizations of certain generalized slices in the affine Grassmannian to each other, this work relates strict chiral quantizations of those same slices. One wonders whether this may have an interpretation in terms of the geometry of the double affine Grassmannian, (see \cite{Braverman:2007dvq}), but we leave this speculation to future work.

Finally, this work is yet another entry into the ever-growing list of geometric free-field realizations. The idea of obtaining free-field realizations by constructing a sheaf of vertex algebras and localizing them onto a global (quasi-)Darboux patch originates in \cite{Beem:2019tfp}. More recently, this viewpoint has been used successfully to provide free-field realizations for several vertex algebras originating from supersymmetric gauge theories---see \cite{Arakawa:2023cki,Coman:2023xcq,Beem:2023dub, Beem:2022univ, Fur} for an incomplete list.

\subsection*{Acknowledgments}

The authors would like to thank Dra\v{z}en Adamovi\'c, Tomoyuki Arakawa, Justine Fasquel, Zachary Fehily, Naoki Genra, Thibault Juillard, Toshiro Kuwabara, Shigenori Nakatsuka, and Yehao Zhou for useful discussions. We would like to especially thank Christopher Beem, whose collaboration in the early stages of this project was invaluable.

We gratefully acknowledge the support of ERC grant \# 864828, and for D.B. the support of the Simons Collaboration - New Structures in Low-dimensional Topology grant.


\section{Preliminaries}

\subsection{Vertex algebras and vertex Poisson algebras}

Write $\ik[\del,z]$ for the global sections of differentials operators on $\IA^1$, and write $\ik[\del]$ for the subalgebra of weakly translation-invariant differential operators.

The category $\ik[\del]$-Mod is symmetric monoidal under $\tensor{\ik}$ with coproduct $\del\mapsto 1\otimes \del + \del\otimes 1$. Therefore, it makes sense to talk about algebras in $\ik[\del]$-Mod. In particular, a unital, commutative, associative algebra in $\ik[\del]$-Mod, is a unital, commutative, associative algebra over $\ik$ equipped with an action of $\del$ by derivations.

\begin{dfn}\label{dfn:Lie star}
	A Lie* algebra  is a $\ik[\del]$-module $L$ equipped with $\IN$-many $\ik$-bilinear operations 
	$$
	_{(n)}: L\otimes L \rightarrow L
	$$
	for $n\in \IN$, satisfying
	\begin{enumerate}[(i)]
		\item $x_{(n)}y =0 $ for some $n\gg 0$ \textit{locality}
		\item $x_{(n)}y = (-1)^{n+1} \sum_{j=0}^\infty \frac{(-1)^j}{j!} \partial^j y_{(n+j)}x$  \textit{skewsymmetry}
		\item $x_{(m)}(y_{(k)}z) - y_{(k)}(x_{(m)}z)  = \sum_{j=0}\begin{pmatrix}m\\ j\end{pmatrix} (x_{(j)}y)_{(m+k-j)}z $ \textit{weak associativity}
		\item $(\del x)_{(n)} y = - n x_{(n-1)}y, (\del x)_{(0)}y=0$ \textit{sesquilinearity}
	\end{enumerate}
\end{dfn}
Elsewhere in the literature, these algebras are known as Lie conformal algebras or vertex Lie algebras. We align our naming conventions with those of \cite{BeilinsonDrinfeld}.

\begin{dfn}
	A vertex Poisson algebra is a pair $(V,\ket{0})$ equipped with a family of products
	$$
	_{(n)}: V\otimes V \rightarrow V
	$$
	for $n\in \IZ$ and $n\ge-1$, such that
	\begin{enumerate}[(i)]
		\item the product $_{(-1)}$ makes $V$ a commutative algebra in $\ik[\del]$-Mod, with unit $\ket{0}\in V$.
		\item $(V,\{ _{(n)}, n\ge 0\})$ is a Lie* algebra 
		\item for any $v,a,b\in V$ and any $n\ge0$, 
		\[
		v_{(n)} a_{(-1)} b = (v_{(n)}a)_{(-1)}b + a_{(-1)} (v_{(n)}b) 
		\]
	\end{enumerate}
\end{dfn}
We call the last condition the \emph{quasi-Leibniz rule}---quasi since it is not symmetric. One should compare this to the characterisation of Poisson algebras as a vector space equipped with a commutative multiplication and a Lie bracket such that the Lie bracket distributes over the commutative product as a derivation.

It is conventional to collect the various products together into a $\lambda$-bracket---a $\ik$-bilinear map $[\cdot_\lambda\cdot] : V\otimes V \rightarrow V[\lambda]$ defined as 
\begin{equation}
	[a_\lambda b] \coloneqq \sum_{n=0}^{\infty} \frac{\lambda^n}{n!} a_{(n)}b~.
\end{equation}

\begin{dfn}
A vertex algebra is a tuple $(V,\ket{0},\{_{(n)}| n\in\IZ\})$, where $V$ is a $\ik[\del]$-module, $\ket{0}\in \ker\del\subset V$ is a distinguished vector called the \emph{vacuum}, and 
	$$
	_{(n)}: V\otimes V \rightarrow V
	$$
is a $\IZ$-family of products on $V$ satisfying
\begin{enumerate}[(i)]
	\item $x_{(n)}\ket{0} = \delta_{n,-1}x$ for any $x\in V$ \textit{vacuum}
	\item $x_{(k)}y =0$ for $k\gg 0$ \textit{locality}
	\item $x_{(n)}y = (-1)^{n+1} \sum_{j=0}^\infty \frac{(-1)^j}{j!} \partial^j y_{(n+j)}x$  \textit{skewsymmetry}
	\item $x_{(m)}(y_{(n)}z) - y_{(m)}(x_{(n)}z)  = \sum_{j=0}\begin{pmatrix}m\\ j\end{pmatrix} (x_{(j)}y)_{(m+n-j)}z $ \textit{weak associativity}
	\item $(\del x)_{(n)} y = - n x_{(n-1)}y$ \textit{sesquilinearity}
\end{enumerate}
for $x,y,z\in V$ and $m,n\in \IZ$.
\end{dfn}

We can collect together the various $_{(n)}$-products into a State-Field Correspondence, 
$$Y(\cdot,z): V\rightarrow \End(V)\Lz$$ 
with $a(z) \equiv Y(a,z) \coloneqq \sum_{n\in \IZ} a_{(n)}z^{-n-1}$. This leads to the Operator Product Expansions (OPEs).
\begin{equation}
	a(z)b(w) \sim \sum_{n=0}^\infty \frac{(a_{(n)}b)(w)}{(z-w)^{n+1}} 
\end{equation}
where the right hand side should be understood as the expansion of the left hand side in the region $|z|>|w|$.

\begin{rem}
	Given a vertex algebra $(V, \ket{0},\{_{(n)}|n\in\IZ\})$, we say that it is generated by a vector subspace $W\subset V$, if any vector $v \in V$, with $v\neq \ket{0}$, can be written as a finite linear combination of monomials of the form $w_{1,(-n_1)}w_{2,(-n_2)}\dots w_{m,(-n_m)}$ where $w_1,w_2,\dots,w_m\in W$ and $n_1,n_2,\dots,n_m>0$. In other words, $V$ is \emph{strongly generated} by any basis of $W$, but we are lax about whether this is a minimal set of generators.
\end{rem}

\subsection{Arc spaces and vertex Poisson algebras}

Let $\ID$ denote the formal disk, \ie, $\ID = {\rm Spf}(\ik\Pt)$. Let $Y$ be a scheme over $\ik$. The functor 
\begin{equation}
\begin{split}
	 {\rm Sch}^{\rm op}_{/\ik} & \rightarrow {\rm Set}\\
	 S  & \mapsto {\rm Hom}(S\times \ID,Y)
\end{split}
\end{equation}
is representable by the scheme $\JJ Y$. Its $\ik$-points are ${\rm Hom}(\ID,Y)$ and we have a natural projection ${\rm ev}: \JJ Y \rightarrow Y$ by evaluating at the closed point of $\ID$. We write $\OO_{\JJ Y}$ for the structure sheaf of $\JJ Y$.

We have a natural evaluation map ${\rm ev}: \JJ Y\twoheadrightarrow Y$, by evaluating a map at the closed point of $\ID$. The fibres of this map can be identified with $\IA^\infty$, and so $\JJ Y$ contracts to $Y$. The pullback ${\rm ev}^*$ provides a natural inclusion $\OO_Y\subset \OO_{\JJ Y}$. Henceforth, the inclusion  $\OO_Y\subset \OO_{\JJ Y}$ always refers to this pullback along ${\rm ev}$.

If $Y$ is Poisson, the global sections $\OO(\JJ Y)$ have the structure of a vertex Poisson algebra, specified by 
\begin{equation}
	[a_\lambda b] = \{a,b\}~,
\end{equation}
for $a,b\in \OO(Y)\subset \OO(\JJ Y)$ and extended to arbitrary sections by $\ik[\del]$-sesquilinearity and the quasi-Leibniz rules. Moreover, $\OO_{\JJ Y}$ is a sheaf of vertex Poisson algebras, see \cite{Arakawa2015:local} for more details on localizing vertex Poisson algebras.

\begin{prop}
	The assignment $Y\mapsto \JJ Y$ is functorial
\end{prop}



Vertex algebras and vertex Poisson algebras admit an operadic construction, see \cite{BeilinsonDrinfeld, Bakalov:2019operads}, and so there are natural cohomology theories controlling their deformations.
Explicit constructions of the relevant cochain complexes can be found in \cite{Bakalov:2019operads, Bakalov:2021classvar,DeSole2013:var}. For brevity, we shall not recall their definitions here. Instead we recall a result on the so-called variational Poisson cohomology, from \cite{BN:def} that is a generalization of \cite{DeSole2013:var}.

\begin{prop}[\cite{BN:def}]\label{thm:BN def variational}
Suppose $Y$ is an affine, symplectic variety. The variational Poisson cohomology of $\OO(\JJ Y)$, with its natural vertex Poisson structure, can be computed as 
\begin{equation}
	H^i_{\rm var}(\OO(\JJ Y)) = 
	\begin{cases}
	H^0_{\rm dR}(Y) \oplus H^1_{\rm dR}( Y) \quad \text{for } i=0 \\
	H^{i+1}_{\rm dR}(Y) \quad \text{for } i>0
	\end{cases}
\end{equation}
\end{prop}
\begin{rem}
We note for future purposes that the zeroth variational Poisson cohomology of $\OO(\JJ Y)$ parameterizes the space of vertex Poisson Casimirs, \ie, all $[a]\in \OO(\JJ Y)/\del \OO(\JJ Y)$ such that
\begin{equation}
	a_{(0)}=0
\end{equation}
as a derivation.
\end{rem}

\subsection{Chiral quantization}

Throughout this subsection, we take $Y$ to be a symplectic variety over $\IC$. Let us introduce the concept of a $\hbar$-adic vertex algebra. Roughly speaking, this is a vertex algebra whose OPEs satisfy locality only up to some power of $\hbar$, \ie, an inverse limit of vertex algebras. We recall the definition in \cite{Li2004:vertexpoisson}:
\begin{dfn}
	A $\hbar$-adic vertex algebra is a tuple $(V,\ket{0},\partial, \{_{(n)}|n\in\IZ\})$ where 
	\begin{enumerate}[(i)]
		\item $V$ is a flat $\ik\Ph$-module, complete in the $\hbar$-adic topology
		\item $\ket{0}\in V$
		\item $\del:V\rightarrow V$ is an endomorphism of $\ik\Ph$-modules;
		\item $ _{(n)}:V\otimes V\rightarrow V$ are $\ik\Ph$-bilinear morphisms
	\end{enumerate}
	such that for each $m\in \IN$, the truncations $(V/\hbar^m V, \ket{0},Y_{\hbar}|_{\hbar^m=0},\partial|_{\hbar^m=0})$ are vertex algebras over $\ik[\hbar]/\hbar^m$.
\end{dfn}

We say that a $\hbar$-adic vertex algebra $(V,\ket{0},Y_\hbar,\partial)$ is \textit{almost commutative} if $(V/\hbar V,\ket{0}, Y_0,\partial)$ is a commutative vertex algebra.
\begin{prop}[\protect{\cite[Proposition 5.6]{Li2004:vertexpoisson}}]
	If $(V,\ket{0},Y_\hbar,\partial)$ is an almost commutative $\hbar$-adic vertex algebra, then the limit $(V/\hbar V)$ has a natural vertex Poisson algebra structure.
\end{prop}

\begin{exm}
	Let $\hf$ be a Lie algebra and let $V^k(\hf)$ denote the universal affine vertex algebra of $\hf$ at some, possibly trivial, level $k\in\IC$. An important example of a $\hbar$-adic vertex algebra is $V_\hbar^k(\hf)$, the $\hbar$-adic analogue of $V^k(\hf)$. This $\hbar$-adic vertex algebra is generated over $\IC\Ph$ by $\hf$ with OPEs
	\begin{equation}
	x(z)y(w) \sim \frac{\hbar [x,y](w)}{z-w} + \frac{\hbar^2 k\langle x,y\rangle}{(z-w)^2}~.
	\end{equation}
	Note also that $V_\hbar^k(\hf)$ is an almost commutative vertex algebra, and $V_\hbar^k(\hf)/\hbar$ is identified with the vertex Poisson algebra $\OO(\JJ\, \hf^*)$.
\end{exm}

We now define quantizations of the vertex Poisson algebra $\OO(\JJ Y)$ and the sheaf of vertex Poisson algebras $\OO_{\JJ Y}$; we use the term quantization for what we call formal quantizations in \cite{BN:def}.

\begin{dfn}\label{dfn:quantization}
	 A \emph{quantization} of $\OO_{\JJ Y}$ is a pair $(\AA,\varphi)$, where
	\begin{enumerate}[(i)]
	\item $\AA$ is a sheaf of almost commutative $\hbar$-adic vertex algebras on $\JJ Y$
	\item $\varphi: \AA/\hbar \AA \xrightarrow{\sim} \OO_{\JJ Y}$ is an isomorphism of sheaves of vertex Poisson algebras
	\end{enumerate}
	Similarly, a \emph{quantization} of the vertex Poisson algebra $\OO(\JJ Y)$ is a pair $(A,\psi)$, where 
	\begin{enumerate}[(i)]
		\item $A$ is an almost commutative $\hbar$-adic vertex algebra
		\item $\varphi: A/\hbar A\xrightarrow{\sim} \OO(\JJ Y)$ is an isomorphism of vertex Poisson algebras.
	\end{enumerate}
\end{dfn}
Note that the global sections of a quantization of $\OO_{\JJ Y}$ give rise to a quantization of $\OO(\JJ Y)$. Like in the associative setting, we also have the notion of a graded quantization.

\begin{dfn}
 Suppose $Y$ has an action of $\IG_m$ that scales the symplectic form with unit weight. We use $\IG_\hbar$ for this distinguished $\IG_m$ action. This lifts to an action of $\IG_\hbar(\OO)$ on $\JJ Y$ but we restrict to the action of the constant arcs $\IG_\hbar\subset \IG_\hbar(\OO)$ and moreover consider the diagonal action under which $\IG_\hbar$ acts also on the formal disc by scaling $\hbar$ with unit weight.

 A \emph{graded} quantization of $\OO_{\JJ Y}$ is a pair $(\AA,\varphi)$ as above together with an $\IG_\hbar$-equivariant structure on the sheaf $\AA$ of $\hbar$-adic vertex algebras such that $\IG_\hbar$ scales $\hbar$ with weight 1 and $\varphi$ is $\IG_\hbar$-equivariant.
 
Similarly, a graded quantization of $\OO(\JJ Y)$ is a pair $(A,\varphi)$ as above together with an action of $\IG_\hbar$ on $A$ by $\hbar$-adic vertex algebra automorphisms such that $\IG_\hbar$ scales $\hbar$ with weight 1 and $\varphi$ is $\IG_\hbar$-equivariant.
\end{dfn}

\begin{thm}[\cite{BN:def}]\label{thm:space of chiral quantization}
	Suppose $Y$ is an affine, symplectic variety such that $\OO(\JJ Y)$ admits a graded quantization, and $H^3_{\rm dR}(Y)$ is concentrated in a single degree with respect to $\bb G_\hbar$-weight. Then the moduli space of graded quantizations of $\OO(\JJ Y)$, up to isomorphism, is a torsor for $H^3_{\rm dR}(Y)$.
\end{thm}

Henceforth, unless specified, we wil use quantization to mean graded quantization.

Suppose that $Y$ has a Hamiltonian action of an algebraic group $H$, with moment map $\Phi: H\rightarrow \hf^*$. As we have remarked, $\OO(\JJ Y)$ has a comoment map $\Phi_\infty: \OO(\JJ \hf^*)\rightarrow \OO(\JJ Y)$ which integrates to an action of $H(\OO)$.

We define a weakly $H(\OO)$-equivariant structure on a quantization $A$ of $\OO(\JJ Y)$ as a map $\rho:H(\OO) \to \textup{Aut}(A)$ such that the map $\varphi$ is $H(\OO)$-equivariant with respect to the induced $H(\OO)$-action on $A/\hbar A$. We define a weakly $H(\OO)$-equivariant quantization $A$ as a quantization together with a choice of weakly $H(\OO)$-equivariant structure.

Similarly, we define a weakly $H(\OO)$-equivariant structure on a quantization $\AA$ of $\OO_{\JJ Y}$ as a map $\rho:H(\OO) \to\textup{Aut}(\AA)$ such that the map $\varphi$ is $H(\OO)$-equivariant with respect to the induced $H(\OO)$-action on $\AA/\hbar \AA$. is weakly $H(\OO)$-equivariant if $\AA$ is a $H(\OO)$-equivariant sheaf and the map $\varphi$ is $H(\OO)$-equivariant.  We define a weakly $H(\OO)$-equivariant quantization $\AA$ as a quantization together with a choice of weakly $H(\OO)$-equivariant structure.

\begin{prop}[\cite{BN:def}]	
Suppose $\OO(\JJ Y)$ admits a weakly $H(\OO)$-equivariant quantization, and that $H^3_{\rm dR}(Y)^{H-inv}$, the cohomology of $H$-invariant forms, is concentrated in a single degree with respect to $\bb G_\hbar$-weight. Then the moduli space of graded, weakly $H(\OO)$-equivariant quantizations is a torsor for $H^3_{\rm dR}(Y)^{H-inv}$.
\end{prop}

A weakly equivariant quantization $A$, of $\OO(\JJ Y)$, is strongly $H(\OO)$-equivariant if it admits a chiral comoment map $\Phi:V_\hbar^k(\hf)\rightarrow A$, a map of $\hbar$-adic vertex algebras such that the action of $U(\hf[\![t]\!])$ integrates to the action of $H(\OO)$. Moreover, we require that $\varphi$ must intertwine   $\Phi_{\hbar=0}$ and the classical comoment map. A weakly equivariant quantization $\AA$ of $\OO_{\JJ Y}$ is strongly equivariant if its global sections are.

We remark that it is not true that a Hamiltonian action of $H(\OO)$ lifts to a strong action on some given quantization of $\OO(\JJ Y)$. Nevertheless, in \cite{BN:def}, we give sufficient conditions for such a lift to exist.

\begin{prop}[\cite{BN:def}]\label{prop:lifting strong actions}
	Suppose the following criteria are met:
	\begin{enumerate}[(i)]
		\item $H$ is unipotent and abelian
		\item the action of $H(\OO)$ is free
		\item the moment map $\Phi_\infty: \JJ Y\rightarrow \JJ \hf^*$ is smooth
	\end{enumerate}
	Then, any graded quantization of $\OO(\JJ Y)$ can be lifted to a strongly $H(\OO)$-equivariant graded quantization. Moreover, such a lift is unique up to strongly $H(\OO)$-equivariant isomorphism.
\end{prop}

We say that a sheaf of almost commutative $\hbar$-adic vertex algebras $\AA$ on $\JJ Y$ is a strict chiral quantization of $Y$ if it is a quantization of $\OO_{\JJ Y}$. In the literature, strict chiral quantization has a slightly different definition. Namely, a strict chiral quantization of $\OO(Y)$ is a vertex algebra (not $\hbar$-adic) $V$ such that ${\rm gr}\, V$, with respect to the Li filtration, is isomorphic to $\OO(\JJ Y)$, as a vertex Poisson algebra.

\subsection{Equivariant Slodowy slices}\label{ssec:Lieprelim}

The spaces we wish to quantize will be the arc spaces of equivariant Slodowy slices. We provide a brief review of the equivariant slices here, mainly to introduce future notation. A more in depth review can be found in \cite{Losev2007:quant}.

Recall that nilpotent coadjoint orbits in $\gf^*\equiv \mathfrak{gl}_N$ are in bijection with the partitions of $N$. Let $\mu$ be such a partition, as above, and let $\chi_\mu=\langle f_\mu , \cdot \rangle$ be an element of the nilpotent coadjoint orbit, $\IO_\mu\subset \gf^*$. 

The orbit $\IO_\mu$ admits a transversal slice $S_\mu$ at $\chi_\mu$---see, \eg, 3.2.19 in \cite{Chriss2010}. There are many such transversal slices at $\chi_\mu$, but a particularly nice and familiar class are the Slodowy slices, which have the following construction. By the Jacobson--Morozov theorem, we may complete $f\equiv f_\mu$ to an $\slf_2$-triple $(h,e,f)$, and the Slodowy slice $S_\mu$ is defined as the image of the affine linear subspace
\begin{equation}
f + \ker\, {\rm ad}_e
\end{equation}
under the Killing isomorphism. Note that $\ker \ad_e$ is always transverse to $\im\, \ad_f=  T_{f} \IO_f$, so that $S_\mu$ defines a transverse slice. Changing the choice of $\slf_2$-triple gives isomorphic slices, related by $G$-conjugation.

Observe that if $\chi=0\in \IO_{[1^N]}$, the $\slf_2$ triple is $(0,0,0)$ and $S_{[1^N]}\cong \gf^*$. Alternatively, if $\chi\in \IO_{[N]}$, \ie, is regular, then $S_{\mu}$ is called the \textit{principal} Slodowy slice and $S_{[1^N]}\cong \mf h/W$.

The transversal slices $S_\mu$ are Poisson, which is more apparent in their definition via Hamiltonian reduction. Given an $\slf_2$-triple, $(h,e,f)$, the operator ${\rm ad}_h$ gives a grading on $\gf$---which we assume is even without loss of generality in type $A$. We let $\gf_{>0}$ denote the Lie subalgebra that is strictly positively graded under ${\rm ad}_h$ and define 
\begin{equation}
	N_\chi \coloneqq {\rm Exp}\, \gf_{>0}~,
\end{equation}
the corresponding unipotent Lie group. Then the coadjoint action of $G$ restricts to a Hamiltonian $N_\chi$-action on $\gf^*$ with moment map  $\Phi_{N_{\chi_\mu}}:\gf^* \to \gf_{>0}^*$ the projection, and we have
\begin{equation}
\begin{split}
	N_\chi \times S_\mu &\xrightarrow{\sim} \{\chi\} + \gf_{\ge0}^* = \mu_{N_\chi}^{-1}(\{\chi\}) \\
	(n,s) &\mapsto {\rm Ad}^*_n s
\end{split} \ ,
\end{equation}
the Gan--Ginzburg isomorphism \cite{GG}. Therefore, we see that
\begin{equation}
	S_\mu \cong \gf^*\red{\chi}N_\chi 
\end{equation}
and thus transverse slice $S_\mu$ inherits the Poisson structure induced by Hamiltonian reduction from the Kirillov--Kostant--Souriau Poisson structure on $\gf^*$.

The action of ${\rm ad}_h$ descends to the transverse slice, giving rise to the Kazhdan grading. On an element $s\in S_\mu$, the action of $t\in\IG_m$ is given by
\begin{equation}\label{eq:Kazhdan}
t\cdot s = t^{1-\tfrac12 \rm{ad}_h} s
\end{equation}
Note that the action, contracts to the fixed point $\chi\in S_\mu$. This action of $\IG_m$ scales the Poisson bracket with weight $-1$. 

Much of our deformation theory makes use of the fact that the underlying space is symplectic and so we want to replace these Poisson transversal slices with a symplectic version. We do so by considering the equivariant Slodowy slices of \cite{Losev2007:quant}.

Viewing $T^*G$ in its \emph{right} trivialization, we have  $T^*G\cong \gf^*\times G$, with left and right moment maps $\Phi_L:\gf^*\times G$ and $\Phi_R:\gf^*\times G \rightarrow G$, where 
\begin{equation}
	\Phi_L(x,g)= x~, \quad \Phi_R(x,g) = {\rm Ad}^*_g x~,
\end{equation}
for $(x,g)\in \gf^*\times G$.

More generally, for any $\chi_\mu\in \IO_\mu$, let $N_{\chi_\mu}$ be as before. The equivariant Slodowy slice is the Hamiltonian reduction
\begin{equation}
	S_{G,\mu}\coloneqq T^*G\red{\chi_\mu} N_{\chi_\mu} ~,
\end{equation}
using the left moment map. Since the left moment map is just projection to $\gf^*$, we see that
\begin{equation}
	S_{G,\mu}\cong G\times S_\mu~.
\end{equation}
hence the name, equivariant Slodowy slice. Since the moment map is flat and its fibers are $N_\chi$-torsors, each reduction is smooth and inherits a symplectic structure from $T^*G$.

Furthermore, since we are in the right trivialization, the right action of $G$ on itself leaves the cotangent fibers invariants. Therefore, $G$ acts freely on the right on each $S_{G,\mu}$ and
\begin{equation}
	S_{G,\mu}/G \cong S_{\mu}~.
\end{equation}

Note that $T^*G$ has a natural action of $\IG_m$ that scales the fibers with weight $1$. The reductions $S_{G,\mu}$ inherit a modified version of this $\IG_m$-action, via the Kazhdan grading---which extends straightforwardly to $\gf^*\times G$. Indeed, this is the distinguished $\IG_\hbar$-action from the preceding section.

\subsection{Inverse Hamiltonian reduction for finite W-algebras}\label{ssec:finite inv ham red}
We end with reviewing some of the main geometric results of \cite{BN:finite}. Let $\mu$ be a partition of $N$, of length $n$---thought of as a dominant coweight of $\PGL_n$. Let $\alpha$ be a positive coroot of $\PGL_n$ such that $\mu+\alpha$ is a dominant coweight, \ie, remains a partition of $N$.

We write $\IG_\alpha$ for the group of translations on $\IA^{{\rm len}\, \alpha}$ and write $\Gf_\alpha\equiv {\rm Lie}\, \IG_\alpha$. Let $\chi_\alpha\in \Gf_\alpha^*$ denote the character $(v_1,\dots,v_{{\rm len}\,\alpha})\mapsto v_1$. Moreover, write $\open{\Gf}\vphantom{\open{\Gf}}^*_{\alpha} = \{ \chi\in \Gf^*\,|\, (1,0\dots,0)\notin {\rm ker}\,\chi \}$.

We identify $\open{T^*}\IG_\alpha\equiv\open{\Gf}\vphantom{\open{\Gf}}^*_\alpha\times \IG_\alpha$ with an open subset of $T^*\IG_\alpha$, thereby giving it the structure of a symplectic variety with a natural Hamiltonian action of $\IG_\alpha$.

As a variety $\open{T^*}\IG_\alpha\cong \IA^{2 {\rm len}\,\alpha-1}\times \IA^\times\cong {\rm Spec}\, \IC[\beta_i,\gamma_i,p,e,e^{-1}\,|\,i=1,\dots,{\rm len}\,\alpha-1]$ and the symplectic form is 
\begin{equation}
	\omega = \frac{1}{e}dp \wedge de + \sum_{i=1}^{\rm len\, \alpha - 1} d\beta_i\wedge d\gamma_i
\end{equation}
The moment map for the Hamiltonian $\IG_\alpha$ action, in these co-ordinates, is given by the projection to the $\beta_i,e$ co-ordinates.

\begin{thm}[\cite{BN:finite}]\label{thm:BN finite main}
The equivariant Slodowy slice $S_{G,\mu}$ has a Hamiltonian action of $\IG_\alpha$ with moment map $\Phi_\alpha:S_{G,\mu}\rightarrow \Gf_\alpha$ such that 
\begin{equation}
	S_{G,\mu+\alpha} \cong S_{G,\mu}\red{(1,\dots,0)}\IG_\alpha~.
\end{equation}
Moreover, this reduction admits an inverse, in the sense that we have a Hamiltonian $\IG_\alpha\times G$-equivariant embedding of symplectic varieties,
\begin{equation}
	{\bf m}: \open{T^*}\IG_\alpha \times S_{G,\mu+\alpha} \rightarrow S_{G,\mu}
\end{equation}
where the action of $\IG_\alpha$ on the left is the natural action on $\open{T^*} \IG_\alpha$ and the $G$ action is the free, right $G$ action on $S_{G,\mu+\alpha}$ and $S_{G,\mu}$.
\end{thm}

%
%

\section{Drinfeld--Sokolov reduction for sheaves of vertex algebras}

For this section we take $H$ to be an algebraic Lie group, with $\hf$ its Lie algebra and $\chi\in \hf^*$ a regular character. Later, we will restrict to the case where $H$ is unipotent. For convenience, we pick some basis $(x_i)$ of $\hf$ such that the structure constants are given by $f_{ij}^k$.

Throughout, $Y$ is taken to be a symplectic variety over $\IC$, with a Hamiltonian action of $H$ and moment map $\Phi:Y\rightarrow \hf^*$. We assume that $\Phi$ is flat in a neighbourhood around $\chi$ and the fibre $\Phi^{-1}(\chi)\subset Y$ has a free action of $H$.

\subsection{Hamiltonian reduction for arc spaces}

Before moving to Drinfeld--Sokolov reduction, we first review the classical limit for arc spaces and vertex Poisson algebras. In usual symplectic geometry, we can define the symplectic reduction of $Y$ at $\chi$ as
\begin{equation}
	Y\red{\chi}H \coloneqq \Phi^{-1}(\chi)/H~.
\end{equation}
From the hypotheses on $Y$, the reduction $Y\red{\chi} H$ is a smooth variety and inherits a symplectic structure from the symplectic form on $Y$.

Let us describe how this construction lifts to the arc space $\JJ Y$. By functoriality, $\JJ Y$ is equipped with an action of the (positive) loop group $H(\OO)$. Moreover, we have a chiral moment map $\Phi_\infty: \JJ Y \rightarrow \JJ \hf^*$, \ie, a $H(\OO)$ equivariant map of vertex Poisson schemes. In particular, we have a chiral comoment map $\Phi^\#_\infty:\OO(\JJ \hf^*) \rightarrow \OO(\JJ Y)$, which is a map of vertex Poisson algebras.

Abusing notation, we write $\chi\in \JJ \hf^*$ to denote the constant map on $\ID$ with image $\chi\in \hf^*$. The fibre $\Phi_\infty^{-1}(\chi)$ is a $H(\OO)$-torsor, since $\Phi^{-1}(\chi)$ is by hypothesis. Thus, we define 
\begin{equation}
	\JJ Y \red{\chi} H(\OO) \coloneqq \Phi_\infty^{-1}(\chi)/H(\OO)~.
\end{equation}

We note that,
\begin{equation}
	\Phi_{\infty}^{-1}(\chi)/ H(\OO) \cong \JJ \bigg( \Phi^{-1}(\chi)/H\bigg)~.
\end{equation}
\subsection{Classical BRST reduction for vertex Poisson algebras}\label{ssec:BRST reduction for VPAs}

Recall the (classical) Clifford algebra $\Lambda(\hf\oplus \hf^*) = \OO(T^* \Pi\hf)$, which is generated by $\xi\in \hf^*$ and $x\in \hf$ in cohomological degrees $-1$ and $1$ respectively, with Poisson bracket
\begin{equation}
	\{\xi,x\} = \langle \xi,x\rangle~,
\end{equation}
Thus, $\JJ\Lambda(\hf\oplus\hf^*)\equiv\OO(\JJ T^*\Pi \hf)$ has the structure of a vertex Poisson algebra. Fixing the basis $(x_i)$ of $\hf$, we can take $(b_i,c_i)$ to be the strong generators of $\JJ\Lambda(\hf\oplus\hf^*)$---with Lambda brackets
\begin{equation}
	[b_i\,_\lambda c_j] = \delta_{ij}
\end{equation}

Let $V$ be a vertex Poisson algebra with a homomorphism of vertex Poisson algebras $\Phi: \OO(\JJ \hf^*)\rightarrow V$ which integrates to an action of $H(\OO)$. Consider the following element of $C^\bullet_{cl}(V,\chi) \equiv V\otimes \JJ \Lambda(\hf\oplus \hf^*)$.
\begin{equation}\label{eq:BRST classical diff}
d_\chi = \sum_{i} (\Phi(x_i) - \chi(x_i))c_i - \tfrac12 \sum_{i,j,k} f_{ij}^k b_kc_ic_j~.
\end{equation}
Then, $C^\bullet_{cl}(V,\chi)$ equipped with the differential $(d_\chi)_{(0)}$ is a chain complex and its cohomology $H^\bullet_{cl}(V,\chi)$, is the classical BRST reduction of $V$ at $\chi$.

\begin{prop}[\cite{Arakawa2015:local}]
Suppose $\Phi: \OO(\JJ \hf^*)\rightarrow V $ is flat as a map of commutative algebras, and that the action of $G(\OO)$ on $V$ is cofree, then for $i\neq 0$,
\begin{equation}
	H^i_{cl}(V,\chi) =0~.
\end{equation}
\end{prop}

This BRST reduction admits a straightforward sheafification, as explained in \cite{Arakawa2015:local}, and we review it here for the case of the structure sheaf $\JJ Y$.

Starting from $\OO_{\JJ Y}$ we form the complex of sheaves%
\footnote{The tensor product is in the category of sheaves over $\JJ Y$, and $\Lambda(\hf^*\oplus \hf)$ is understood, by abuse of notation, to be the constant sheaf over $\JJ Y$.}
of vertex Poisson algebras $\CC_{cl}(\OO_{\JJ Y},\chi)=\OO_{\JJ Y}\otimes \JJ \Lambda(\hf^*\oplus \hf)$, equipped with differential given by the zero mode of the global section \eqref{eq:BRST classical diff}.

It is clear to see that on any open $U\in \JJ Y$, 
$$
\CC_{cl}^\bullet(\OO_{\JJ Y},\chi)(U) \cong C^\bullet_{cl}\big(\OO_{\JJ Y}(U),\chi\big)
$$
We write $\HH^{\bullet}_{ cl}(\OO_{\JJ Y},\chi)$ for the cohomology of this sheaf and note that it is supported on $\Phi_\infty^{-1}(\chi)\subset \JJ Y$ (see the Theorem below). Writing $p_\infty: \Phi_\infty^{-1}(\chi)\rightarrow \Phi_{\infty}^{-1}(\chi)/H(\OO)$, we have a sheaf of vertex Poisson algebras on $\JJ Y\red{\chi}H(\OO)$, given by
\begin{equation}
	\OO_{\JJ Y} \red{\chi} H(\OO) \coloneqq p_{\infty *} \HH^{0}_{cl}(\OO_{\JJ Y}\otimes\JJ\Lambda(\hf\oplus\hf^*))
\end{equation}
which is the BRST reduction of $\OO_{\JJ Y}$.

\begin{thm}[\protect{\cite[Theorem 2.3.3.1]{Arakawa2015:local}}]\label{thm:AKM vertex Poisson}
Suppose $Y$ is as above, then
\begin{enumerate}[(i)]
	\item for each affine open $U\subset \JJ Y$, 
	$$H^{i}\big(\CC_{cl}(\OO_{\JJ Y},\chi)(U)\big)\cong H^i_{cl}(\OO_{\JJ Y}(U),\chi)=0$$ 
	for $i\neq0$
	\item for each affine open $U\subset \JJ Y$, such that $U\cap \Phi_\infty^{-1}(\chi)$ is empty, $$H^{i}\big((\CC_{cl}(\OO_{\JJ Y},\chi)(U)\big)\cong H^i_{cl}(\OO_{\JJ Y}(U),\chi)=0$$ 
	for all $i$
	\item the support of $\HH^{i}_{cl}(\OO_{\JJ Y},\chi) $ is $\Phi_{\infty}^{-1}(\chi)$ and $\HH^{i}_{cl}(\OO_{\JJ Y},\chi)=0$ for $i\neq0$
	\item for each affine open $U\subset \JJ Y$, $$\HH^\bullet_{cl}\big(\OO_{\JJ Y},\chi\big)(U) \cong H^\bullet\big(\CC_{cl}(\OO_{\JJ Y},\chi)(U)\big)\cong H_{cl}^\bullet(\OO_{\JJ Y}(U),\chi)$$
	\item There is an isomorphism $$ \OO_{\JJ Y}\red{\chi}H(\OO)\equiv \OO_{\JJ (Y\red{\chi}H)}$$
	of sheaves of vertex Poisson algebras on $\JJ Y\red{\chi}H(\OO)$.
\end{enumerate}
\end{thm}


\subsection{Drinfeld--Sokolov reduction}\label{ssec:DS reduction}

Now, we lift this construction to the quantized setting. 

Write ${\rm Cl}(\hf)$ for the vertex (super)algebra generated by $\hf^*\oplus \hf$ in cohomological degrees $-1$ and $1$ respectively, with OPEs
\begin{equation}
	\xi(z)x(w) \sim \frac{ \langle\xi,x\rangle}{z-w}~.
\end{equation}
Write, ${\rm Cl}_\hbar(\hf)$ for the $\hbar$-adic version of ${\rm Cl}(\hf)$, with OPEs
\begin{equation}
	\xi(z)x(w) \sim \frac{\hbar \langle\xi,x\rangle}{z-w}~.
\end{equation}

The vertex algebra ${\rm Cl}(\hf)$ is, perhaps more familiarly, known as the $bc$-ghost system associated with $\hf$. The vertex Poisson algebra $\JJ \Lambda^\bullet(\hf^*\oplus \hf)$ is, therefore, its classical analogue. Moreover, ${\rm Cl}_{\hbar}(\hf)|_{\hbar=0} \cong \JJ\Lambda(\hf\oplus\hf^*)$ as vertex Poisson algebras, \ie, it is a quantization.

 The vertex algebra ${\rm Cl}_\hbar(\hf\oplus \hf^*)$ is strongly generated by the odd fields $b_i(z)$ and $c_i(w)$ of cohomological degree $-1$ and $+1$, with OPEs
\begin{equation}
	b_i(z) c_j(w) \sim \frac{\delta_{ij}}{z-w}
\end{equation}


To continue, as in \cite{Arakawa2015:local}, we have to make a technical restriction on $H$. 
\emph{For the chiral setting we  take the group $H$ is to be unipotent}. 
\\
This special case of BRST reduction, is called Drinfeld--Sokolov reduction.

Given a vertex algebra $V$ with a comoment map $\Phi: V(\hf)\rightarrow V$ that integrates to an action of $H(\OO)$, we may form the complex $C_{\rm DS}^\bullet(V,\chi) = V\otimes {\rm Cl}(\hf)$ equipped with the differential $(d^{ch}_{\chi})_{(0)}$, where $d^{ch}_{\chi}$ is now given by
\begin{equation}\label{eq:diff for hbar DS}
d^{ch}_\chi = \sum_{i} (\Phi(x_i) - \chi(x_i))_{(-1)}c_i - \tfrac12 \sum_{i,j,k} f_{ij}^k (b_{k,(-1)}(c_{i,(-1)}c_j)~.
\end{equation}
The cohomology $H^\bullet_{\rm DS,\hbar}(V,\chi)$ is the Drinfeld--Sokolov reduction of $V$ at $\chi$.

Now, suppose $V$ is a $\hbar$-adic vertex algebra equipped with a morphism $\Phi_{ch}:V_\hbar(\hf) \rightarrow V$, that integrates to an action of $H(\OO)$. Then, we form the complex $C^\bullet_{\rm DS,\hbar}(V,\chi)\equiv V\otimes {\rm Cl}_{\hbar}(\hf)$, where the tensor product is over $\ik\Ph$ and completed. The differential is given by 
\begin{equation}
\frac{1}{\hbar} d_\chi^{ch}
\end{equation}
where $d_{\chi}^{ch}$ is the same expression as \eqref{eq:diff for hbar DS}, now interpreted as elements of $V\otimes {\rm Cl}_{\hbar}(\hf)$.
The cohomology $H^\bullet_{\rm DS,\hbar}(V,\chi)$ is the $\hbar$-adic Drinfeld--Sokolov reduction of $V$ at $\chi$.

\subsection{W-algebras}

As a motivating example, we review the construction of $\hbar$-adic W-algebras by Drinfeld--Sokolov reduction of $V_\hbar^k(\gf)$.

Let $\mu$ be a partition of $N$, and consider $\chi_\mu\in\IO_\mu$ as in Section \ref{ssec:Lieprelim}. 

We have a natural comoment map by inclusion, $V_\hbar(\gf_{>0})\rightarrow V_\hbar^k(\gf)$, which integrates to an action of $N_\chi$. Form the $\hbar$-adic DS reduction complex $C_{{\rm DS},\hbar}(V^k(\gf),\chi)$.

The cohomology of this chain complex is concentrated in degree zero \cite{Arakawa2015:local}, and we have 
\begin{equation}
	\WW_{\hbar,\mu} \coloneqq H^0_{\rm DS,\hbar}(V_\hbar^k(\gf), \chi_\mu)~.
\end{equation}

By \cite[Theorem 3.2.6.1]{Arakawa2015:local}, we may recover the ordinary affine W-algebras, as defined in say \cite{Arakawa:2017Walg,Kac:2003DS}, by
\begin{equation}
	 H^0_{\rm DS}(V^{k}(\gf),\chi_\mu)=\WW_{\mu} \cong \bigg(\WW_{\hbar,\mu}\tensor{\ik\Ph}\ik(\!(\hbar)\!)\bigg)^{\IG_\hbar}~.
\end{equation}

\subsection{Drinfeld--Sokolov reduction for sheaves}


In Section \ref{ssec:BRST reduction for VPAs}, we described how to sheafify BRST reduction for the special case of the structure sheaf $\OO_{\JJ Y}$. Here we describe how to lift this to a quantization of the structure sheaf, \ie, a sheaf of $\hbar$-adic vertex algebras.

Let $\AA$ be a strongly $H(\OO)$-equivariant quantization of $\OO_{\JJ Y}$ with comoment map $\Phi_{ ch}:V_\hbar(\hf)\rightarrow \AA$. We form the complex $\CC^\bullet_{\rm DS}(\AA,\chi) = \AA\otimes {\rm Cl}_\hbar(\hf)$, where the (completed) tensor product is taken over sheaves of $\ik\Ph$-modules. Once again, we have a distinguished global section
\begin{equation}
	d^{ch}_\chi(z) = \sum_i ((\Phi_{ ch}(x_i)(z) - \chi(x_i))c_i(z)  - \tfrac12 \sum_{i,j,k} f_{ij}^k :b_k(z) c_i(z)c_j(z):
\end{equation}
where the products are now understood to be normally ordered. The operator $\tfrac{1}{\hbar}
d^{ ch}_{\chi,(0)}$ defines a differential on $\AA\otimes {\rm Cl}_\hbar(\hf)$, \textit{c.f.}, \eqref{eq:diff for hbar DS}. We denote the cohomology sheaf by $\HH^{\bullet}(\AA\otimes {\rm Cl}_\hbar(\hf\oplus\hf^*))$, and note that it is supported on $\Phi_\infty^{-1}(\chi)$ (see the theorem below). We have a candidate for a quantization of $\JJ Y\red{\chi}  H(\OO)$, given by 
\begin{equation}
\AA\red{\chi}H(\OO) \coloneqq p_{\infty*} \HH^{0}(\AA,\chi)~.	
\end{equation}

\begin{thm}[\protect{\cite[Theorem 2.3.5.1]{Arakawa2015:local}}]\label{thm:AKM chiral}
	Let $Y$ be a smooth symplectic variety with a Hamiltonian action of a connected unipotent group $H$. Let $\Phi:Y\rightarrow \hf$ be the moment map, and let $\AA$ be a strongly $H(\OO)$-equivariant quantization of $\OO_{\JJ Y}$.
	\begin{enumerate}[(i)]
		\item For each affine open $U\subset \JJ Y$, $H^{i}\big((\CC_{\rm DS}^{ch}(\AA,\chi)(U)\big) =0$ for $i<0$
		\item For each affine open $U\subset \JJ Y$, such that $U\cap \Phi_\infty^{-1}(\chi)$ is empty, $H^{i}\big((\CC_{\rm DS}^{ch}(\AA,\chi))(U)\big) =0$ for all $i$
		\item The support of $\HH^i(\AA,\chi)$ is $\Phi_\infty^{-1}(\chi)$ and $\HH^i(\AA,\chi)=0$ for $i\neq0$
		\item For each affine open $U\subset \JJ Y$, $\HH^i(\AA,\chi)(U) \cong H^{i}\big({\rm \CC}^{ch}_{\rm DS}(\AA,\chi))(U)\big)$
		\item The DS reduction $\AA\red{\chi} H(\OO)	$ is a quantization of $\OO_{\JJ Y\red{\chi}H(\OO)}$
	\end{enumerate}
\end{thm}

\begin{rem}
Note that even if $\AA$ is a graded quantization of $\OO_{\JJ Y}$, the reduction will generically fail to be a graded quantization of $\OO_{\JJ Y\red{\chi} H}$, since the character $\chi$ is not a $\IG_\hbar$ fixed point.	
\end{rem}


\section{Sheaves of equivariant W-algebras}

Write $\DD_{\hbar,G,k}^{ch}$ for the sheaf of $\hbar$-adic cdos---at level $k\in\IC$---on $\JJ T^*G$, see Appendix \ref{app:cdos} for a construction. For now, we note that its global sections are 
\begin{equation}
	\DD_{\hbar,k}^{ch}(G)
\end{equation}
the $\hbar$-adic chiral differential operators on the group $G$ at level $k\in \IC$---a $\hbar$-adic analogue of the cdos on $G$ described in \cite{Gorbounov:2000cdosG}. 

We want to construct a quantization of $\JJ S_{G,\mu}$ for any coweight $\mu$. Starting with the trivial coweight $\mu=[1^N]$, $\JJ S_{G,[1^N]}\cong \JJ T^*G$. We have a one-parameter family of quantizations of this arc space, given by the sheaf of $\hbar$-adic chiral differential operators $ \DD^{ ch}_{\hbar,G,k}$, for a choice of level, $k\in \ik$. 

The global sections, $\DD^{ch}_{\hbar,k}(G)$, has two commuting current subalgebras at dual levels, \ie, there are two chiral comoment maps 
\begin{equation}\label{eq:G comoment maps}
	\Phi_L : V^k_\hbar(\gf) \rightarrow \DD^{ch}_{\hbar,G,k} \text{ and } \Phi_R: V_\hbar^{k-2h^\vee}(\gf)\rightarrow \DD^{ch}_{\hbar,G,k}~,
\end{equation}
where $h^\vee$ is the dual Coxeter number. Fix a representative $\chi_\mu\in \IO_\mu$ and let $N_\chi$ be defined as in Section \ref{ssec:Lieprelim}. By restricting the left comoment map $\Phi_L$ we have a comoment map for the $N_\chi(\OO)$ action on $\JJ T^*G$.

Then, we construct a putative 1-parameter family of quantizations of $\JJ S_{G,\mu}$, given by the DS reductions
\begin{equation}
\AA_{\hbar,G,\mu}^k \coloneqq	 \DD^{ch}_{\hbar,G,k}\red{\chi_\mu} N_{\chi_\mu}(\OO)~,
\end{equation}
for $k\in \IC$.

Noting that the $N_{\chi_\mu}$ action on $T^*G$ satisfies the conditions of Theorem \ref{thm:AKM chiral}, we see that $\AA_{\hbar,G,\mu}$ is indeed a graded quantization of $\JJ S_{G,\mu}$. Write
	\begin{equation}
		\mathbf{W}_{\hbar,G,\mu}^k\coloneqq\Gamma(\JJ S_{G,\mu},\AA_{\hbar,G,\mu}^k),
	\end{equation}
where the notation is meant to be evocative of the equivariant affine W-algebras introduced in \cite{Arakawa:2018egx}.

\begin{prop}
	The quantizations $\AA^k_{\hbar,G,\mu}$ are strongly $G(\OO)$-equivariant for the right $G(\OO)$-action on $\JJ S_{G,\mu}$.
\end{prop}
\begin{proof}
 Note that this is true when $\mu=[1^N]$, with the comoment map given in \eqref{eq:G comoment maps}. 
 Since the right $G(\OO)$-action commutes with the action of $N_{\chi_\mu}(\OO)$ that is used in the reduction, we have a comoment map 
 \[
 H^0_{\rm DS}(\Phi_R): V^{k-2h^\vee}(\gf)\rightarrow {\bf W}^k_{\hbar,G,\mu}
 \]
 by functoriality of DS-reduction.
\end{proof}

Let us describe how to recover the equivariant affine W-algebras of \textit{loc.\ cit.}
We write 
\begin{equation}
\big({\bf W}^k_{\hbar,G,\mu}\big)_{G(\OO)\times\IG_\hbar-{\rm fin}} \coloneqq \{v \in{\bf W}^k_{\hbar,G,\mu}\,|\, v \text{ is in a fin.\ dim.\ rep.\ of } G(\OO)\times \IG_\hbar\}~.
\end{equation}

\begin{lem}
The subalgebra $\big({\bf W}^k_{\hbar,G,\mu}\big)_{G(\OO)\times\IG_\hbar-{\rm fin}}$ is a vertex algebra over $\ik[\hbar]$.
\end{lem}

Write 
\begin{equation}
{\rm Fil}({\bf W}^k_{\hbar,G,\mu}) = \bigg(\big({\bf W}^k_{\hbar,G,\mu}\big)_{G(\OO)\times\IG_\hbar-{\rm fin}}\tensor{\ik[\hbar]} \ik[\hbar,\hbar^{-1}]\bigg)^{\IG_\hbar}~,
\end{equation}
for the filtered vertex algebra corresponding to $\big({\bf W}^k_{\hbar,G,\mu}\big)_{G(\OO)\times\IG_\hbar-{\rm fin}}$ under the Rees correspondence. The ascending filtration $F_\hbar^i\,{\rm Fil}({\bf W}^k_{\hbar,G,\mu})$ is defined by
\begin{equation}
	F_\hbar^i\,{\rm Fil}({\bf W}^k_{\hbar,G,\mu}) = \{v \in {\rm Fil}({\bf W}^k_{\hbar,G,\mu})\,| \, \hbar^i v  \in {\bf W}^k_{\hbar,G,\mu}\}~.
\end{equation}
Note that this filtration is separated, but in general unbounded below and ${\rm Fil}({\bf W}^k_{\hbar,G,\mu})$ is \textit{not completed} with respect to this filtration.

\begin{prop}\label{prop:finite subalgebra is equiv W algebra}
	$$
	\bigg(\big({\bf W}^k_{\hbar,G,\mu}\big)_{G(\OO)\times \IG_\hbar-{\rm fin}}\tensor{\IC[\hbar]}\IC[\hbar^{-1},\hbar]\bigg)^{\IG_\hbar}\cong {\bf W}^k_{G,\mu}
	$$
	where on the right ${\bf W}^k_{G,\mu}$ is the equivariant affine W-algebra of \cite{Arakawa:2018egx}.
\end{prop}
\begin{proof}
We follow closely the argument of \cite{Dodd:2025LocalWalg}. Note that for $\mu=[1^n]$, this is the statement for cdos on $G$, which we know to be true by Lemma \ref{lem:Rees for h-adic cdos}.

Now suppose $\mu> [1^n]$ and recall the complex,
$$
C^\bullet_{\rm DS,\hbar}(\DD^{ch}_{\hbar,k}(G),\chi_\mu)=\DD^{ch}_\hbar(G) \otimes {\rm Cl}_\hbar(\nf_{\chi_\mu})
$$
Since the left and right $G$ moment maps commute, the $G$-finite subalgebra gives rise to a subcomplex. Moreover, by equipping the complex with the Kazhdan $\IG_\hbar$-action---under which the differential is $\IG_\hbar$-equivariant---we get the subcomplex $C^\bullet_{\rm DS,\hbar}(\DD_{\hbar,k}^{ch}(G),\mu)_{G(\OO)\times \IG_\hbar-fin}$. 

Write 
$$
{\rm Fil}\big(C^{\bullet}_{\rm DS}(\DD^{ch}_{\hbar,k}(G),\chi_\mu)\big) \equiv \big(C^\bullet_{\rm DS}(\DD^{ch}_{\hbar,k}(G),\chi_\mu)_{G(\OO)\times \IG_\hbar-fin}\tensor{\IC[\hbar]}\IC[\hbar,\hbar^{-1}]\big)^{\IG_\hbar}~.
$$
Recall from \cite{Arakawa:2018egx} that the equivariant affine W-algebra ${\bf W}^k_{G,\mu}$ is defined as the cohomology of the complex $C^\bullet_{\rm DS}(\DD^{ch}_{k}(G),\chi_\mu)$, the non $\hbar$-adic Drinfeld--Sokolov reduction complex. See Section 6 of \textit{loc.\ cit.} for more details.
We have an inclusion of $\ik$-vertex algebras, $	C^{\bullet}_{\rm DS}(\DD^{ch}_{k}(G),\chi_\mu)\hookrightarrow {\rm Fil}\big(C^{\bullet}_{\rm DS}(\DD^{ch}_{\hbar,k}(G),\chi_\mu)\big)$, which sends 
some $v\in C^\bullet(\DD^{ch}_k(G),\chi_\mu)$ to 
\begin{equation}
v\mapsto \hbar^{-|v|} v~,
\end{equation}
where $| \cdot |$ is the weight of $v$ under the Kazhdan grading. 

By pulling back the natural filtration on ${\rm Fil}\big(C^{\bullet}_{\rm DS}(\DD^{ch}_{\hbar,k}(G),\chi_\mu)\big)$ we obtain an ascending filtration on $C^\bullet_{\rm DS}(\DD_k^{ch}(G),\chi_\mu)$. 

This filtration is conjugate to the canonical, descending Li filtration on $C^\bullet_{\rm DS}(\DD_k^{ch}(G),\chi_\mu)$. Taking cohomology and passing to associated graded we obtain an embedding of vertex Poisson algebras
\begin{equation}
	\OO(\JJ S_{G,\mu})\rightarrow \OO(\JJ S_{G,\mu})~,
\end{equation}
which must be an isomorphism. Therefore, the embedding $	C^{\bullet}_{\rm DS}(\DD^{ch}_{k}(G),\chi_\mu)\hookrightarrow {\rm Fil}\big(C^{\bullet}_{\rm DS}(\DD^{ch}_{\hbar,k}(G),\chi_\mu)\big)$ must be a quasi-isomorphism, and ${\rm Fil}({\bf W}^k_{\hbar,G,\mu})\cong {\bf W}^k_{G,\mu}$ as desired.
\end{proof}

\begin{lem}\label{lem:V and W are dual pairs}
	The subalgebras $V ^{k-2h^\vee}(\gf)\hookrightarrow {\bf W}^k_{ G,\mu}$ and $\WW^k_{ \mu}(\gf)\hookrightarrow {\bf W}^k_{ G,\mu}$ are dual pairs, \ie,
	\begin{equation}
		\WW^k_{ \mu}(\gf) = {\rm Comm}_{{\bf W}^k_{ G,\mu}}\big(V^{k-2h^\vee}(\gf)\big)~, \text{ and }~ V ^{k-2h^\vee} (\gf)= {\rm Comm}_{{\bf W}^k_{ G,\mu}}\big(\WW^k_{ \mu}(\gf)\big)~.  
	\end{equation}
	Moreover, in the $\hbar$-adic setting we have
	\begin{equation}
		\WW^k_{\hbar, \mu}(\gf) = {\rm Comm}_{{\bf W}^k_{\hbar, G,\mu}}\big(V^{k-2h^\vee}_\hbar(\gf)\big)~, \text{ and }~ V_\hbar ^{k-2h^\vee} (\gf)= {\rm Comm}_{{\bf W}^k_{\hbar, G,\mu}}\big(\WW^k_{\hbar, \mu}(\gf)\big)~.  
	\end{equation}
\end{lem}

\begin{proof}
	The first equality is \cite[Proposition 6.5]{Arakawa:2018egx}. Therefore, we have an inclusion, $V^{k-2h^\vee}(\gf)\subset {\rm Comm}_{{\bf W}^k_{ G,\mu}}\big(\WW^k_{ \mu}\big)$.

	Now, from Proposition 6.6 of \textit{loc.\ cit.}, we see that ${\bf W}^k_{G,\mu}$ has a filtration whose associative graded is 
	\begin{equation}
	\bigoplus_{\lambda\in {\rm \Lambda^+({\rm GL}_N)} } 	H^0_{\rm DS}(\IV_\lambda)\otimes D(\IV_{\lambda^*})
	\end{equation}
	as $\WW^k_{\mu}\otimes \widehat{\gf}_{k-2h^\vee}$ modules. As graded $\widehat{\gf}_{k-2h^\vee}$ modules, we identify 
	\begin{equation}\label{eq:sizeofcommutant}
	{\rm Comm}_{{\bf W}^k_{ G,\mu}}\big(\WW^k_{ \mu}(\gf)\big)\cong 1\otimes D(\IV_0)\cong V^{k-2h^\vee}(\gf)~.
	\end{equation}
	The inclusion $V^{k-2h^\vee}(\gf)\hookrightarrow {\rm Comm}_{{\bf W}^k_{ G,\mu}}\big(\WW^k_{ \mu}(\gf)\big)$ is compatible with this filtration and upon taking associated graded we have an isomorphism from \eqref{eq:sizeofcommutant}. Thus, we see that $V^{k-2h^\vee}(\gf)\cong {\rm Comm}_{{\bf W}^k_{ G,\mu}}\big(\WW^k_{ \mu}(\gf)\big)$ as desired.

	The $\hbar$-adic result then follows by completing with respect to the filtration and using Proposition \ref{prop:finite subalgebra is equiv W algebra}.
\end{proof}

\begin{prop}\label{prop:rigidity for equiv W algebras}
	Any quantization of $\OO(\JJ S_{G,\mu})$ is isomorphic to ${\bf W}^k_{\hbar,G,\mu}$ for some $k\in \IC$.
\end{prop}
\begin{proof}
	From Theorem \ref{thm:space of chiral quantization}, we see that the dimension of the moduli space of graded quantizations is a torsor for $H^3_{\rm dR}(S_{G,\mu})=\IC$. The equivariant W-algebras, ${\bf W}^k_{\hbar,G,\mu}$, constitute a one-parameter family of quantizations. Thus, any quantization is isomorphic to ${\bf W}^k_{\hbar,G,\mu}$, for some $k\in\IC$.
\end{proof}

From Lemma \ref{lem:V and W are dual pairs} and Proposition 6.5 of \cite{Arakawa:2018egx}, we note that
\begin{equation}
	\WW^k_{\hbar,\mu}\cong ({\mathbf W}^k_{G,\hbar,\mu})^{G(\OO)}
\end{equation}
To decomplete and recover the non-$\hbar$-adic $\WW$-algebra we can use the usual trick of passing to the asymptotic algebra and taking $\IG_\hbar$-invariants
\begin{equation}
	\WW^k_\mu \cong \big(\WW^k_{\hbar,\mu}\tensor{\ik\Ph}\ik(\!(\hbar)\!) \big)^{\IG_\hbar}~,
\end{equation}
as in \cite{Arakawa2015:local}.


The other ingredient in inverse Hamiltonian reduction is a quantization of $\JJ\open{T^*}\IG_\alpha$, as introduced in Section \ref{ssec:finite inv ham red}.

The translation group $\IG_\alpha \cong \IA^{\ell(\alpha)}$ has an obvious chiral quantization given by $\hbar$-adic chiral differential operators, \ie, the $\hbar$-adic $\beta\gamma$-system of rank $\ell(\alpha)$. We may localise the $\beta\gamma$-system on $\JJ \open{T^*}\IG_\alpha$, to obtain $\DD_{\hbar,\alpha}^{ch}$, whose global sections $\DD_{\hbar}^{ch}(\JJ\open{T^*}\IG_\alpha)$ are strongly generated by $p,e^\pm,\beta_i,\gamma_i$ subject to the relation
\begin{equation}
	:e^+e^-:(z) = \id
\end{equation}
and with singular OPEs given by
\begin{equation}
	p(z) e^\pm(w) \sim \frac{\pm \hbar e^\pm(w)}{z-w}~, \text{ and } ~ \beta_i(z) \gamma_j(w) \sim \frac{\delta_{ij}\hbar}{z-w}~.
\end{equation}

\begin{prop}\label{prop:cdos are unique}
	The microlocal cdos $ \DD^{ch}_{\hbar}(\JJ\open{T^*\IG_\alpha})$ is the unique quantization of $\JJ\open{T^*}\IG_\alpha$ 
\end{prop}
\begin{proof}
	This follows from Theorem \ref{thm:space of chiral quantization}, noting that $H^3_{\rm dR}(\open{T^*}\IG_\alpha)=0$.
\end{proof}

The Hamiltonian action of $\IG_\alpha(\OO)$ on $\JJ \open{T^*} \IG_\alpha$ lifts to a strong $\IG_\alpha(\OO)$-action on $\DD_{\hbar}^{ch}(\JJ\open{T^*}\IG_\alpha)$ with chiral comoment map
$	\Phi: V(\Gf_\alpha) \rightarrow \DD_{\hbar}^{ch}(\JJ\open{T^*}\IG_\alpha)$ given by the inclusion of the subalgebra generated by the $e,\beta_i$. 

\begin{prop}\label{prop:cdos are unique strong}
	The localized cdos $\DD_{\hbar}^{ch}(\JJ\open{T^*}\IG_\alpha)$ with the given comoment map are the unique strongly $\IG_\alpha(\OO)$-equivariant quantization of $\OO(\JJ \open{T^*}\IG_\alpha)$. Any other quantization is isomorphic to this one by a strongly $\IG_\alpha$-equivariant isomorphism.
\end{prop}
\begin{proof}
	Follows from Propositions \ref{prop:cdos are unique} and \ref{prop:lifting strong actions}, since the $\IG_\alpha(\OO)$ action is free and the comoment map is smooth.
\end{proof}


\section{Inverse Hamiltonian reduction for arc spaces}


Fix $H$ to be an abelian Lie group with Lie algebra $\hf$. Let $Y$ be a symplectic variety with a Hamiltonian action of $H$ and moment map $\Phi:Y\rightarrow \hf^*$. Let $\chi\in\hf^*$ be a regular character and consider the Hamiltonian reduction $Y\red{\chi} H$.

\begin{dfn}
	An inverse Hamiltonian reduction (IHR) for the reduction $\JJ Y\red{\chi} H(\OO)$ is a tuple $(U, f)$ where
	\begin{enumerate}[(i)]
		\item $U\subset \JJ \hf^*$ is an open vertex Poisson subscheme, such that $\chi\in {\rm ev}(U)\subset \hf^*$.
		\item $f: \JJ \big(H \times U\big)\times \JJ \big(Y\red{\chi}H\big) \xrightarrow{\sim} \JJ Y \fibre{\Phi_\infty,\JJ\hf^*} U$ is a strongly $ H(\OO)$-equivariant isomorphism of vertex Poisson schemes.
	\end{enumerate}
\end{dfn}

\begin{prop}
	The IHR data for $Y$ defines IHR data for $\JJ Y$.
\end{prop}
\begin{proof}
	Follows from functoriality of $\JJ$.
\end{proof}

\begin{prop}\label{prop:IHR implies reduction commutes with quantisation}
Suppose $H$ is unipotent and $(U,f)$ is IHR data for $\JJ Y\red{\chi} \JJ H$ and let $\AA$ be a strongly $H(\OO)$-equivariant quantization of $\JJ Y$. Then,
\begin{enumerate}[(i)]
	\item $\AA|_{\Phi_\infty^{-1}(U)}\red{\chi} H(\OO) \cong \AA\red{\chi} H(\OO)$
	\item $\AA\red{\chi} \JJ H$ is a quantization of $\JJ ( Y\red{\chi} H)$.
\end{enumerate}
\end{prop}
\begin{proof}
The existence of an IHR $(U,f)$ for $\JJ Y \red{\chi} H(\OO)$ means that the hypotheses of Theorem \ref{thm:AKM chiral} are satisfied and both statements follow from the results therein.
\end{proof}

%
%

%
%
\section{Inverse Hamiltonian reduction in type A}

Throughout, fix $\mu$ to be a partition of $N$ of length $n$. We think of $\mu$ as a positive dominant coweight of $\PGL_n$ in the usual way. Let $\alpha$ be a positive coroot of $\PGL_n$ such that $\mu+\alpha$ is positive dominant, \ie, $\mu+\alpha$ remains a partition of $N$. We denote by $\IG_\alpha$ the additive group of rank equal to the length of $\alpha$.

\subsection{Inverse Hamiltonian reduction for arcs to equivariant Slodowy slices}
From the main result of \cite{BN:finite}, recalled in Theorem \ref{thm:BN finite main}, we have the following immediate corollaries for the arc spaces of equivariant Slodowy slices.

\begin{cor}\label{cor:Result for classical W algebras}
 The arc space $\JJ S_{G,\mu}$ admits a vertex Poisson action of $\IG_\alpha(\OO)$, such that 
 \[
		\JJ S_{G,\mu}\red{\chi_{\alpha}}\IG_\alpha(\OO) \cong \JJ S_{G,\mu+\alpha}
 \]
 moreover, the reduction admits an inverse, $(\JJ \open{\Gf^*}, {\bf m}_\infty)$, \ie, ${\bf m}_\infty$ is a vertex Poisson $\IG_\alpha(\OO)\times G(\OO)$---equivariant embedding of vertex Poisson schemes
 \[
	{\bf m}_\infty: \JJ \open{T^*}\IG_\alpha \times \JJ S_{G,\mu+\alpha} \hookrightarrow \JJ S_{G,\mu}~,
 \]
 where the action of $\IG_\alpha(\OO)$ on the left is the natural action on $\JJ \open{T^*}\IG_\alpha$ and the $G(\OO)$-action is the free, right action on $\JJ S_{G,\mu+\alpha}$ and $\JJ S_{G,\mu}$.
\end{cor}
\begin{proof}
	This follows straightforwardly by functoriality of the arc space functor and Theorem \ref{thm:BN finite main}.
\end{proof}

On functions, this implies that $\OO(\JJ S_{G,\mu})$ has a strong action of $\IG_\alpha(\OO)$, \ie, there is a vertex Poisson algebra morphism $\Phi_\alpha: \OO(\JJ \Gf_\alpha^*)\rightarrow \OO(S_{G,\mu})$ that integrates to an action of $\IG_\alpha(\OO)$. Moreover, 
\begin{equation}
	\OO(\JJ S_{G,\mu}) \red{\chi_\alpha} \IG_\alpha(\OO) \cong \OO(\JJ S_{G,\mu+\alpha})~.
\end{equation}
The second result implies that we have the following embedding of vertex Poisson algebras
\begin{equation}
	\JJ {\bf m}^\#: \OO(\JJ S_{G,\mu})\hookrightarrow \OO(\JJ S_{G,\mu+\alpha})\otimes \OO(\JJ \open{T^*} \IG_\alpha)~.
\end{equation}
Moreover, this embedding is strongly $\IG_\alpha(\OO)\times G(\OO)$-equivariant, where $\IG_\alpha(\OO)$ acts on the left as described, and on the right acts solely on the $\OO(\JJ \open{T^*}\IG_\alpha)$ factor, and $G(\OO)$ acts on $\OO(\JJ S_{G,\mu})$ and $\OO(\JJ S_{G,\mu+\alpha})$ via the natural right action.

\subsection{Existence of strong and weak actions}

Much like in the finite case, we rely on a deformation theory argument from \cite{BN:def} to justify the existence of a strong $\IG_\alpha(\OO)$ action on $\WW_{\hbar,G,\mu}$.
\begin{lem}\label{lem:moment map is smooth and free}
	The comoment map $\Phi_\infty: \JJ S_{G,\mu}\rightarrow \JJ \Gf_\alpha^*$ is smooth and the action of $\IG_\alpha(\OO)$ on $\JJ S_{G,\mu}$ is free.
\end{lem}
\begin{proof}
	This follows from the observations in \cite{BN:finite} Section, \ie, the $\IG_\alpha$ action on $S_{G,\mu}$ is free, and that the moment map $\Phi_\alpha: S_{G,\mu}\rightarrow \Gf^*_\alpha$ is smooth. Both properties are preserved when passing to arc spaces.
\end{proof}
Thus, the criteria for lifting the action of $\IG_\alpha$ are met.

\begin{prop}\label{prop:S_G have strong actions}
Every quantization of $\OO(\JJ S_{G,\mu})$ can be lifted to a unique, up to strong $\IG_\alpha(\OO)$-equivariant isomorphism, strongly $\IG_\alpha(\OO)$-equivariant quantization. 

In particular for any $k\in\IC$, there exists a chiral comoment map $\Phi:V(\Gf_\alpha) \rightarrow {\bf W}^k_{\hbar,G,\mu}$ that integrates to an action of $\IG_\alpha(\OO)$. 
 
Any strongly $\IG_\alpha(\OO)$-equivariant quantization of $\OO(\JJ S_{G,\mu})$ is, therefore, isomorphic to ${\bf W}^k_{\hbar,G,\mu}$ and the isomorphism can be chosen to be strongly $\IG_\alpha(\OO)$-equivariant.
\end{prop}
\begin{proof}
	From Lemma \ref{lem:moment map is smooth and free}, the criteria of Proposition \ref{prop:lifting strong actions} are met, which implies the result.
\end{proof}

\subsection{Inverse Hamiltonian reduction for affine W-algebras}

To begin, we establish a version of reduction by stages, \ie, a quantum version of the first result in Corollary \ref{cor:Result for classical W algebras}.

Much like the equivariant Slodowy slices themselves, the IHR chart admits a one parameter family of quantizations.

\begin{prop}\label{prop:quant of IHR chart}
	The moduli space of quantizations of $\OO(\open{\JJ T^*\IG_\alpha})\otimes \OO( \JJ S_{G,\mu+\alpha})$ that are weakly $G(\OO)$-equivariant and strongly $\IG_\alpha$-equivariant is one dimensional and any such quantization is isomorphic to 
	\begin{equation}
		\DD^{ch}_\hbar(\open{T^*}\IG_\alpha)\otimes {\bf W}^k_{\hbar,G,\mu}
	\end{equation}
	for some $k\in \IC$, with the isomorphism being weakly $G(\OO)$-equivariant and strongly $\IG_\alpha$-equivariant.
\end{prop}
\begin{proof}
	For any $k\in \IC$,  $\DD^{ch}_\hbar(\open{T^*}\IG_\alpha)\otimes {\bf W}^k_{\hbar,G,\mu}$ is a weakly $G(\OO)$-equivariant and strongly $\IG_\alpha(\OO)$-equivariant, graded quantization of $\OO(\open{\JJ T^*\IG_\alpha})\otimes \OO( \JJ S_{G,\mu+\alpha})$---with the natural $\IG_\alpha(\OO)$-equivariant structure on $\DD^{ch}_\hbar(\open{T^*}\IG_\alpha)$ and the natural $G(\OO)$-equivariant structure on $ {\bf W}^k_{\hbar,G,\mu}$.

	The moduli space of weakly $G(\OO)$-equivariant quantizations is a torsor for $H^3_{\rm dR}(\open{T^*}\IG_\alpha \times S_{G,\mu})^{G-{\rm inv}}= \IC$. Therefore, each $k\in \IC$ is a representative for an isomorphism class of weakly $G(\OO)$-equivariant quantizations and any quantization of $\OO\open{\JJ T^*\IG_\alpha})\otimes \OO( \JJ S_{G,\mu+\alpha})$ is isomorphic to $	\DD^{ch}_\hbar(\open{T^*}\IG_\alpha)\otimes {\bf W}^k_{\hbar,G,\mu}$, with a weakly $G(\OO)$-equivariant isomorphism.

	Moreover, from Proposition \ref{prop:cdos are unique strong}, we see that this isomorphism can be made strongly $\IG_\alpha(\OO)$-equivariant too.
\end{proof}

\begin{thm}\label{thm:reduction by stages}
We have the following isomorphism of $\hbar$-adic vertex algebras
\begin{equation}
	{\bf W}^{k}_{\hbar,G,\mu}\red{\chi_{\alpha}}\IG_\alpha(\OO) \cong {\bf W}^k_{\hbar,G,\mu+\alpha}
\end{equation}
\end{thm}
\begin{proof}
Since $S_{G,\mu}$ admits an IHR for such an $\alpha$ by Theorem \cite{BN:finite}, Proposition \ref{prop:IHR implies reduction commutes with quantisation} tells us that $\AA^k_{\hbar,G,\mu}\red{\chi_\alpha}\IG_\alpha(\OO)$ is a quantization of $\JJ S_{G,\mu+\alpha}$.

Now by Proposition \ref{prop:rigidity for equiv W algebras}, we note that this implies 
$$
\Gamma(\JJ S_{G,\mu+\alpha},\AA^k_{\hbar,G,\mu}\red{\chi_\alpha}\IG_\alpha) = {\bf W}^{k}_{\hbar,G,\mu}\red{\chi_{\alpha}}\JJ\IG_\alpha \cong {\bf W}^{k'}_{\hbar,G,\mu+\alpha}~,
$$
for some $k'\in\ik$. What remains is to justify that $k'=k$.

By functoriality, the inclusion $V^{k-2h^\vee}(\gf)\hookrightarrow{\bf W}^k_{\hbar,G,\mu}$ descends to an inclusion $V^{k-2h^\vee}(\gf)\hookrightarrow {\bf W}^{k'}_{\hbar,G,\mu+\alpha}$. Moreover, the image of this embedding commutes with the $\WW^{k}_{\hbar,\mu+\alpha}$ subalgebra. Thus, from Lemma \ref{lem:V and W are dual pairs} this gives an embedding $V^{k-2h^\vee}_{\hbar}(\gf)\hookrightarrow V_\hbar^{k'-2h^\vee}(\gf)$, which implies $k'=k$ as desired.
\end{proof}

\begin{prop}\label{prop:rigidity iso of quantisations}
There exists a weakly $G(\OO)$-equivariant, strongly $\IG_\alpha$-equivariant isomorphism of graded quantizations
\begin{equation}
	\psi_{\mu,\alpha}: \Gamma\big(\JJ S_{G,\mu+\alpha}\times \JJ \open{T^*}\IG_\alpha ,{\bf m}_\infty^*(\AA^k_{\hbar,G,\mu})\big)\xrightarrow{\sim} {\bf W}^k_{\hbar,G,\mu+\alpha}\otimes \DD_\hbar^{ch}(\open{T^*}\IG_\alpha)
\end{equation}
\end{prop}
\begin{proof}
The sheaf, ${\bf m}_\infty^*(\AA^k_{\hbar,G,\mu})$ is a quantization of $\JJ S_{G,\mu+\alpha}\times \JJ\open{T^*}\IG_\alpha$ and so by Proposition \ref{prop:quant of IHR chart} its global sections must be isomorphic to ${\bf W}^{k'}_{\hbar,G,\mu+\alpha}\otimes \DD_\hbar^{ch}(\open{T^*}\IG_\alpha)$
for some $k'\in \IC$. What remains is to show $k'=k$.

Form the cohomology sheaf, $\HH^\bullet( {\bf m_\infty}^*(\AA^k_{\hbar,G,\mu})\otimes{\rm Cl}_\hbar(\hf\oplus\hf^*))$. Since the action of $\IG_\alpha(\OO)$ satisfies the hypotheses of Theorem \ref{thm:AKM chiral}, we note that 
\begin{equation}\nonumber
\begin{split}
\Gamma\big(\JJ S_{G,\mu+\alpha}\times \JJ \open{T^*}\IG_\alpha, \HH^\bullet( {\bf m}^*\AA^k_{\hbar,G,\mu}\otimes{\rm Cl}_\hbar(\hf\oplus\hf^*))\big) &\cong H^\bullet\big({\bf W}^{k'}_{\hbar,G,\mu+\alpha}\otimes \DD_\hbar^{ch}(\open{T^*}\IG_\alpha)\otimes{\rm Cl}_\hbar(\hf\oplus\hf^*)\big)\\
&\cong {\bf W}^{k'}_{\hbar,G,\mu+\alpha}\otimes H^\bullet\big(\DD_\hbar^{ch}(\open{T^*}\IG_\alpha)\otimes{\rm Cl}_\hbar(\hf\oplus\hf^*)\big)\\
&\cong {\bf W}^{k'}_{\hbar,G,\mu+\alpha}
\end{split}
\end{equation}
The support of $\HH^\bullet( \AA^k_{\hbar,G,\mu}\otimes{\rm Cl}_\hbar(\hf\oplus\hf^*))$ is contained within the inverse Hamiltonian reduction open set, and so 
\[
	\Gamma\big(\JJ S_{G,\mu+\alpha}\times \JJ \open{T^*}\IG_\alpha, \HH^\bullet( {\bf m}^*\AA^k_{\hbar,G,\mu}\otimes{\rm Cl}_\hbar(\hf\oplus\hf^*))\big) \cong \Gamma\big(\JJ S_{G,\mu},\HH^\bullet( \AA^k_{\hbar,G,\mu}\otimes{\rm Cl}_\hbar(\hf\oplus\hf^*))\big) ~.
\]
But by Theorem \ref{thm:reduction by stages} we must have that $\Gamma\big(\JJ S_{G,\mu},\HH^\bullet( \AA^k_{\hbar,G,\mu}\otimes{\rm Cl}_\hbar(\hf\oplus\hf^*))\big)\cong  {\bf W}^{k}_{\hbar,G,\mu+\alpha}$, implying $k'=k$.
\end{proof}

\begin{thm}\label{thm:IHR embedding}
We have a weakly $G(\OO)$-equivariant and strongly $\IG_\alpha(\OO)$-equivariant embedding of $\hbar$-adic vertex algebras
\begin{equation}\label{eq:IHR embedding}
 \Psi_{\mu,\alpha}:{\bf W}^k_{\hbar,G,\mu} \hookrightarrow \DD^{ch}_\hbar(\open{T^*}\IG_\alpha)\otimes {\bf W}^k_{\hbar,G,\mu+\alpha}~,
\end{equation}
where the $\IG_\alpha(\OO)$ action on the left is as in Proposition \ref{prop:S_G have strong actions} and on the right it is the natural action on the localized cdos. The $G(\OO)$ action is the natural right action on ${\bf W}^k_{\hbar,G,\mu}$ and ${\bf W}^k_{\hbar,G,\mu+\alpha}$ respectively.

Similarly, we have a strongly $\IG_\alpha(\OO)$-equivariant embedding of $\hbar$-adic vertex algebras
\begin{equation}
	\Psi_{\mu,\alpha}:\WW_{\hbar,\mu}^k \hookrightarrow \DD^{ch}_\hbar(\open{T^*}\IG_\alpha)\otimes \WW^k_{\hbar,\mu+\alpha}~.
\end{equation}

\end{thm}

\begin{proof}
	The map $\Psi_{\mu,\alpha}$ is obtained by composing the natural restriction map $\Gamma\big(\JJ S_{G,\mu}, \AA^k_{\hbar,G,\mu}\big) \hookrightarrow \Gamma\big( \JJ S_{G,\mu+\alpha}\times \JJ \open{T^*}\IG_\alpha\big)$ with the rigidity isomorphism of Proposition \ref{prop:rigidity iso of quantisations}, both of which have the desired equivariance properties. The second embedding follows from taking $G(\OO)$-invariants on both sides.
\end{proof}


%
\begin{prop}\label{ref:prop embedding is strong}
The embedding $\Psi_{\mu,\alpha}$ is strongly $G(\OO)$-equivariant.
\end{prop}

\begin{proof}
Write $\Phi_1:V^k_\hbar(\gf)\rightarrow {\bf W}^k_{\hbar,G,\mu}$ and $\Phi_2:V^k_{\hbar}(\gf)\rightarrow {\bf W}^k_{\hbar, G,\mu+\alpha}$ for the two comoment maps. Let $x\in \gf$, then 
\begin{equation}
	\Psi_{\mu,\alpha}(\Phi_1(x))(z) = \Phi_2(x)(z) + A(z)
\end{equation}
where $A(z)$ is some field to be determined.  Since $\Psi_{\mu,\alpha}$ is weakly $G(\OO)$-equivariant, it must intertwine the action (by derivations) of $U(\gf[t])$. This action is realized by the positive modes of $\Phi_1(x)$ and $\Phi_2(x)$, on the left and right hand side respectively. 

Since $\Psi_{\mu,\alpha}$ intertwines the $U(\gf[t])$ action, this implies that the OPE of $A(z)$ with any field in the image of $\Psi_{\mu,\alpha}$ must be regular. 

Suppose $A(z) $ is $ O(\hbar^m)$ for some $m\ge 0$, then $\widetilde{A} = \tfrac{1}{\hbar^m}A$ also has regular OPEs with any field in the image of $\Psi_{\mu,\alpha}$.  Taking the classical limit, this implies that $\widetilde{A}|_{\hbar=0}\in \OO(\JJ (S_{G,\mu+\alpha}\times \open{T^*}\IG_\alpha))$ has vanishing $\lambda$-bracket with any element in the $\OO(\JJ S_{G,\mu})$ subalgebra.

In particular, this implies that $\widetilde{A}|_{\hbar=0}$ is a vertex Poisson Casimir, \ie, $[\widetilde{A}|_{\hbar=0}\,_\lambda-]|_{\lambda=0}$ vanishes. The Casimirs, up to total derivatives, are classified by the zeroth variational Poisson cohomology (see \cite{DeSole2013:var}). From Theorem \ref{thm:BN def variational}, this corresponds to $H^1_{\rm dR}( S_{G,\mu}\times \open{T^*}\IG_\alpha)\oplus H^0_{\rm dR}(S_{G,\mu}\times \open{T^*}\IG_\alpha)$. Since none of these classes transform appropriately under $G$, we conclude that $\widetilde{A}|_{\hbar=0}$ must be a total derivative.

To be concrete, set $\widetilde{A}_{\hbar=0} = \del a_1$  for some $a_1\in \OO(\JJ (S_{G,\mu+\alpha}\times \open{T^*}\IG_\alpha))$. We note that $[a_1\, _\lambda -] = \lambda [\widetilde{A}_{\hbar=0}\, _\lambda]$ by sesquilinearity. Therefore, $a_1$ must be a Casimir and, repeating the same argument, we conclude that $a_1 = \del a_2$ for some $a_2\in \OO(\JJ (S_{G,\mu+\alpha}\times \open{T^*}\IG_\alpha))$. Repeating this argument, we conclude that $\widetilde{A}|_{\hbar=0}=0$, and therefore $A = 0$. Thus $\Psi_{\mu,\alpha}(\Phi_1(x))(z) = \Phi_2(x)(z)$ for any $x\in \gf$, as desired.
\end{proof}
 
Let us speak a little on how to recover inverse Hamiltonian reduction for the usual (non $\hbar$-adic) $\WW$-algebras and their equivariant counterparts.

Since $\Psi_{\mu,\alpha}$ is $G(\OO)$-equivariant, we may pass to the $G(\OO)\times \IG_\hbar$-integrable subalgebra of both sides. Then, Proposition \ref{prop:finite subalgebra is equiv W algebra} leads to 
\begin{cor}
	We have strongly $G(\OO)\times \IG_\alpha(\OO)$-equivariant embeddings of vertex algebras
	\begin{equation}
		\Psi_{\mu,\alpha}:{\bf W}^k_{G,\mu} \hookrightarrow {\bf W}^k_{G,\mu+\alpha}\widehat{\otimes}\hat{\DD}^{ch}(\open{T^*}\IG_\alpha)
	\end{equation}
	and by taking $G(\OO)$-invariants, we have strongly $\IG_\alpha(\OO)$-equivariant embeddings
	\begin{equation}
		\Psi_{\mu,\alpha}: \WW^k_{\mu} \hookrightarrow \WW^k_{\mu+\alpha}\widehat{\otimes}\hat{\DD}^{ch}(\open{T^*}\IG_\alpha)
	\end{equation}
	where the actions are as in Theorem \ref{thm:IHR embedding}.
\end{cor}

\subsection{Inverse Hamiltonian reduction for the Kazhdan--Lusztig category}

The ($\hbar$-adic) Kazhdan--Lusztig category $\KL_k$ at level $k$ corresponds to modules of $V_\hbar^k(\gf)$, such that the action of the positive modes integrates to an action of $G(\OO)$. 


Let $M\in \KL_k$, and let $\chi_\mu$ and $N_{\chi_\mu}$ be as in Section \ref{ssec:Lieprelim}. By restriction, $M$ has an integrable action of $V(\nf_{\chi})$. Analogously to Section \ref{ssec:DS reduction}, we can construct a $\WW_{\mu}$ module as the DS reduction $H^0_{\rm DS}(M,\chi_\mu)$.

Given any $M\in \KL_k$, the product $M\otimes {\bf W}^k_{\hbar,G,\mu}$ is an object $\KL_{-2h^\vee}$ with the $V^{-2h^\vee}_\hbar(\gf)$ action coming from the diagonal action on $M\otimes {\bf W}^k_{\hbar,G,\mu}$---noting that ${\bf W}^k_{\hbar,G,\mu}$ has a $V^{\hbar,-k-2h^\vee}(\gf)$ subalgebra. We recall the following result from \cite{Arakawa:2018egx}:

\begin{prop}\label{prop:BRST equivariant W alg}
 Let $M$ be an object in $\KL_k$, then the relative semi-infinite cohomology 
 $$\Hi{\bullet}(\widehat{\gf}_{-2h^\vee},\gf, {\bf W}^{k}_{\hbar,G,\mu}\otimes M)$$ is concentrated in degree zero and 
 \begin{equation}
 \Hi{0}	(\widehat{\gf}_{-2h^\vee},\gf, {\bf W}^{k}_{\hbar,G,\mu}\otimes M) \cong H^0_{\rm DS}(\mu, M)~,
 \end{equation}
 as modules over $\WW^{k}_{\hbar,\mu}$, where on the left hand side the action comes from the $\WW^{k}_{\hbar,\mu}$ subalgebra of ${\bf W}^k_{\hbar,G,\mu}$ and is the natural action on the right hand side.

 Moreover, if $M$ is a vertex algebra object in $\KL$, then this is an isomorphism of vertex algebras.
\end{prop}

In essence, this result allows us to generalize inverse Hamiltonian reduction to more general objects in $\KL_k$.

\begin{prop}
We have a restriction functor 
	\begin{equation}
	\begin{split}
		 {\rm IHR}_\alpha:\WW^k_{\hbar,\mu+\alpha}{\rm - Mod} &\rightarrow \WW^k_{\hbar,\mu}{\rm-Mod}  \\
		 M &\mapsto {\rm Res}^{\WW^k_{\hbar,\mu}}_{\WW^k_{\hbar,\mu+\alpha}\otimes \DD_{\hbar}^{ch}} \big( M \otimes \DD_{\hbar}^{ch}(\open{T^*}\IG_\alpha)\big)
	 \end{split}
	\end{equation}
\end{prop}
\begin{proof}
	Follows from Theorem \ref{thm:IHR embedding}.
\end{proof}

\begin{thm}
	Let $M\in {\rm KL}_k$, and denote by $M_\mu$ the DS reduction $H^0_{\rm DS}(M,\mu)$. Then we have natural embeddings of $\WW^k_{\hbar,\mu}$ modules given by 
	\begin{equation}
		M_\mu \hookrightarrow {\rm IHR}_{\alpha}(M_{\mu+\alpha})~,
	\end{equation}
	In particular, if $M$ is a vertex algebra object in $\KL_k$, this is an embedding of vertex algebras.
\end{thm}
\begin{proof}
We start with the isomorphism of $\WW_{\hbar,\mu}$-modules 
\begin{equation}
M_\mu \cong \Hi{\bullet}(\widehat{\gf}_{-2h^\vee}, \gf, {\bf W}^k_{G,\mu}\otimes M)
\end{equation}
from Proposition \ref{prop:BRST equivariant W alg}. The IHR map $\Psi_{\mu,\alpha}$ provides an inclusion of complexes 
$$
C^{\tfrac{\infty}{2}+\bullet}(\widehat{\gf}_{-2h^\vee},\gf,{\bf W}^k_{G,\mu}\otimes M)\hookrightarrow C^{\tfrac{\infty}{2}+\bullet}(\widehat{\gf}_{-2h^\vee},\gf,{\bf W}^k_{G,\mu+\alpha}\otimes \DD^{ch}(\open{T^*}\IG_\alpha)\otimes M)~,
$$
since it is strongly $G(\OO)$-equivariant. Therefore, it induces an inclusion on cohomology and 
$$
{\rm H}^{\tfrac{\infty}{2}+\bullet}(\widehat{\gf}_{-2h^\vee},\gf,{\bf W}^k_{G,\mu}\otimes M)\hookrightarrow {\rm H}^{\tfrac{\infty}{2}+\bullet}(\widehat{\gf}_{-2h^\vee},\gf,{\bf W}^k_{G,\mu+\alpha}\otimes M)\otimes \DD^{ch}(\open{T^*}\IG_\alpha)~,
$$
and using Proposition \ref{prop:BRST equivariant W alg}, this gives the desired isomorphism.
\end{proof}

\appendix

\section{Microlocalization of chiral differential operators}\label{app:cdos}

Suppose $Y$ is a smooth, affine, $\ik$-scheme such that the second Chern class of $Y$ vanishes, \ie, $Y$ admits a sheaf of chiral differential operators. Importantly, this sheaf $\DD^{ch}_Y$ is only a sheaf over $Y$, but we wish to construct a microlocal version of $\DD^{ch}_Y$ that is a sheaf over $\JJ T^* Y$. In this appendix, we go over an abstract construction of this microlocal sheaf.


\subsection{Chiral envelope of a Lie* algebra}

Recall the definition of a Lie* algebra from Definition \ref{dfn:Lie star}. Given any $\IN$-graded vertex algebra, we can forget the regular terms in the OPE to obtain an $\IN$-graded Lie* algebra, giving rise to a forgetful functor from the category of vertex algebras to that of Lie* algebras. This forgetful functor admits a left-adjoint, the chiral (often called vertex) envelope.

Given an $\IN$-graded Lie* algebra $L$, note that $[\cdot,\cdot]:L\otimes L \rightarrow L$ given by 
\begin{equation}\label{eq:liealgstruct}
	[x,y] = \sum_{j=0}^{\infty} (-1)^j \partial^j (x_{(j)}y)~,
\end{equation}
defines an honest Lie bracket on $L$. Write $U(L)$ for the universal enveloping algebra of $L$ thought of as a Lie algebra under $[\cdot,\cdot]$. We extend the action of $\ik[\partial]$ to $U(L)$ by the Leibniz rule. We have a natural morphism of $\ik$-modules $i:L\hookrightarrow U(L)$, coming from the composition $L\hookrightarrow T(L) \xrightarrow{p} U(L)$, where $T(L)$ is the free tensor algebra on $L$ and $p$ is the natural projection to the universal enveloping algebra.

Given a Lie* algebra, the vector space ${\rm Sym}\, L$ has the natural structure of an $\IN$-graded vertex Poisson algebra, with $\lambda$-bracket
\begin{equation}
	[a_\lambda b] = \sum_{k\in\IN} \frac{\lambda^k}{k!}a_{(k)}b~,
\end{equation}
on $a,n\in L\subset {\rm Sym}\, L$ and extended to arbitrary elements by the left and right quasi-Leibniz rules. As usual, $U(L)$ admits a PBW basis of lexicographically ordered monomials and we have the usual symmetrization map
\begin{equation}
	{\rm Sym}\, L \xrightarrow{\sigma} U(L)~.
\end{equation}

\begin{prop}
	There exists a unique $\IN$-graded vertex algebra structure on  $U(L)$ such that for all $a\in L\subset U(L)$,
	and $v \in T(L)$,
	\begin{equation}\label{eq:normal ordering on envelope}
	p(a\otimes v) =i(a)_{(-1)}p(v)
	\end{equation}
  We write $V(L)$ for the $\ik$-module $U(L)$, when emphasizing its vertex algebra structure.
\end{prop}
We can think of $V(L)$, therefore, as a nonassociative deformation of $U(L)$---akin to how $U(L)$ is a noncommutative deformation of ${\rm Sym}\, L$. Moreover, we have a canonical isomorphism of vector spaces $\psi:U(L)\xrightarrow{\sim} V(L)$, defined recursively on the PBW basis of $U(L)$ by \eqref{eq:normal ordering on envelope}.

\begin{rem}\label{rem:bidiff operators}
The OPEs in $V(L)$ admit a rather nice description. The field $a(z)\equiv Y(a,z)$ for any $a\in L\subset V(L)$ is of the form
\begin{equation}
	Y(a,z) = \sum_{m\in\IZ} a_{(m)}z^{-m-1}~,
\end{equation}
This fixes the OPEs between all elements of $L\subset V(L)$. To compute OPEs of composites, we can use the left and right Wick rules. Now, since all $_{(n)}$ in $L$ are bidifferential operators, this guarantees that all $_{(n)}$-products in $V(L)$ are some finite order bidifferential operators---albeit highly complicated. For a nice exposition in the case where $L$ is a finitely generated $\ik[\del]$-module see \cite{Castellan:2023jnu}.
\end{rem}

Composing the non-associative deformation $\psi$ with the symmetrization map gives a map 
\begin{equation}
	{\rm NO}: {\rm Sym}\, L \xrightarrow{\sim} V(L)~.
\end{equation}
One should think of this as giving rise to a normal-ordering prescription.

One can repeat this construction starting with a Lie* algebroid instead.
\begin{dfn}
	Given a unital, commutative algebra $A\in\ik[\del]-$Mod, a Lie* $A$-algebroid is an $A$-module $L$ equipped with a Lie* bracket $\{_{(n)}, n\in \IN\}$ and an anchor map $\tau: L\rightarrow {\rm Der}(A)$, satisfying
	\begin{enumerate}[(i)]
		\item $\tau(a x)(b) = a\tau(x)(b)$ for any $a,b\in A$ and $x\in L$
		\item $x_{(n)}(a y) = \delta_{n,0}\tau(x)(a)y + a x_{(n)}y$ for any $a\in A$ and $x,y\in L$.
	\end{enumerate}
\end{dfn}
Given a Lie* $A$-algebroid $L$, the relative symmetric algebra ${\rm Sym}_A^\bullet\, L$ has a unique vertex Poisson algebra structure by extending the Lie* bracket using the left and right Leibniz rules. We define $U(L)$ as the quotient of $T(L) = \bigoplus_n L^{\tensor{A} n}$, by the ideal defined by \eqref{eq:liealgstruct}.

\subsection{Chiral differential operators as chiral envelopes}


Let $Y$ be a smooth, affine variety over $\IC$ with trivial second Chern class. Let $c\in H^3_{\rm dR}(Y)$, such a choice fixes (up to isomorphism) a particular cdo over $Y$. For brevity, we suppress this choice in future notation.

Let $L_Y$ denote the following extension 
\begin{equation}
	0\rightarrow \OO(\JJ Y)\hookrightarrow L_Y \twoheadrightarrow \Theta(\JJ Y) \rightarrow 0
\end{equation}
of the global vector fields on $\JJ Y$. A choice of $c$ endows $L_Y$ with the structure of a Lie* $\OO(\JJ Y)$-algebroid, with OPEs given by
\begin{equation}
\begin{split}
 x(z) y(w)    &\sim 	\frac{[x,y](w) - c(x,y)(w)  }{z-w} ~,\\
 x(z) a(w)  &\sim \frac{x(a)(w)}{z-w}
\end{split}
\end{equation}
for $x,y\in \Theta(Y)\subset \Theta(\JJ Y)$ and $a\in \OO(Y)\subset \OO(\JJ Y)$. Moreover, $L_Y$ is $\IN$-graded with $L_{Y}(0)= \OO(Y)$ and $L_{Y}(1)= \Omega^1(Y) \oplus \Theta(Y)$, where $\Omega^1(Y) \cong \OO(Y)\cdot \del \OO(Y)$.

The various $_{(n)}$ products of the Lie* algebra are, therefore, bidifferential operators on the $\Theta(Y)$ and $\OO(Y)$ variables. One should interpret Borcherd's identities as relations satisfied by these bidifferential operators.

By definition, a cdo, $\DD^{ch}(Y)$ is the chiral envelope of $L_Y$. In particular, this means that the OPEs in $\DD^{ch}(Y)$ may be computed using the left and right Wick rules, with each Wick contraction a bidifferential operator.

Let us also define a $\hbar$-adic version of the preceding construction. Let $L_{\hbar,Y}$ denote the $\hbar$-adic version of $L_Y$ with OPEs
\begin{equation}
\begin{split}
 x(z) y(w)    &\sim 	\frac{\hbar [x,y](w) - \hbar^2 c(x,y)(w) }{z-w}  ~,\\
 x(z) f(w)  &\sim \frac{\hbar x(f)(w)}{z-w}
\end{split}
\end{equation}
for $f\in \OO(Y)$ and $x\in \Theta(Y)$. The group $\IG_\hbar$ acts here by Lie* algebra automorphisms, scaling $x\in \Theta(Y)$ and leaving $\OO(\JJ Y)$ invariant. Denote the chiral envelope of $L_{\hbar,Y}$ by $\DD_{\hbar}^{ch}(Y)$---the $\hbar$-adic cdos. Analogously, we have a $\hbar$-adic normal ordering prescription 
\begin{equation}
	{\rm NO}_\hbar: {\rm Sym}\, L_{\hbar,Y} \xrightarrow{\sim} \DD_\hbar^{ch}(Y)~,
\end{equation}
which is now a $\IG_\hbar$-equivariant isomorphism of $\ik\Ph$-modules. Hereafter, ${\rm Sym}$ should be understood as the relative symmetric algebra with respect to $\OO(\JJ Y)$.

Write $L_Y\Ph$ for the $\hbar$-adic completion of $L_Y\otimes\IC \Ph$, we have a natural embedding of $\hbar$-adic Lie* algebras, $\sigma_\hbar: L_{\hbar,Y} \hookrightarrow L_{Y} \Ph$, given by
\begin{equation}
	\sigma_\hbar(f) = f \text{ and } \sigma_{\hbar}(x) = \hbar x~,
\end{equation}
for $f\in \OO(\JJ Y)$ and $x\in \Theta(Y)$, and extended by $\ik[\del]$-linearity. This embedding is $\IG_\hbar$-equivariant, where on the right $\IG_\hbar$ acts by scaling only $\hbar$. This induces an $\IG_\hbar$-equivariant embedding of $\hbar$-adic vertex algebras $\DD_\hbar^{ch}(Y)\hookrightarrow \DD^{ch}(Y)\Ph$, again denoted by $\sigma_\hbar$. Note that $\sigma_\hbar$ becomes an isomorphism after inverting $\hbar$.

\subsection{Microlocalizing cdos}

The $\ik\Ph$-module isomorphism ${\rm NO}_\hbar$ can be leveraged to endow ${\rm Sym}\, L_\hbar(Y)$ with the structure of a $\hbar$-adic vertex algebra. The $_{(n)}$ products are defined for $a,b\in {\rm Sym}\, L_\hbar(Y)$ as
\begin{equation}
	a_{(n)} b \coloneqq {\rm NO}_\hbar^{-1}({\rm NO}_\hbar(a)_{(n)}{\rm NO}_\hbar(b))~,
\end{equation}
analogously to the usual star products for quantizations of a Poisson algebra. Moreover, from Remark \ref{rem:bidiff operators} and the left right Wick rules, we note that 
\begin{equation}
	a_{(n)}b = \sum_{k=1}^\infty \hbar^k P^k_{(n)}(a,b)
\end{equation}
where $P^k_{(n)}$ are finite order bidifferential operators. 

\begin{lem}\label{lem:NOisiso}
	With the given $_{(n)}$-products, ${\rm NO}_\hbar$ lifts to an isomorphism of $\hbar$-adic vertex algebras.
\end{lem}

To construct the microlocalization of $\DD^{ch}_\hbar(Y)$ we can now follow the general logic of \cite{Arakawa2015:local}. Note that ${\rm Sym}\, L_{\hbar,Y}$ is naturally a module over $\OO(\JJ T^* Y)$ and thus defines a quasicoherent sheaf of $\ik \Ph$-modules over $\JJ T^* Y$---call this sheaf  $\DD^{ch}_{\hbar, Y}$.

Let $f\in \OO(\JJ T^* Y)$ be some function and let $U_f$ denote the basic Zariski open defined by $f\neq 0$. We want to endow $\DD^{ch}_{\hbar,Y}(U_f) = ({\rm Sym}\, L_{\hbar,Y})[f^{-1}]$ with the structure of an $\hbar$-adic vertex algebra. Since the $_{(n)}$-products of ${\rm Sym}\, L_{\hbar,Y}$ are bidifferential operators, they naturally extend to the localisation $({\rm Sym}\, L_{\hbar,Y})[f^{-1}]$, essentially via the quotient rule. Moreover, Borcherd's rules continue to be satisfied since the bidifferential operators $P^{k}_{(n)}(\cdot,\cdot)$ satisfy them. Furthermore, since the bidifferential operators are independent of the localization, the vertex algebra structures glue appropriately and $\DD^{ch}_{\hbar,Y}$ is a sheaf of $\hbar$-adic vertex algebras.

It follows, from Lemma \ref{lem:NOisiso} that 
\begin{equation}
	\Gamma(\JJ T^*Y,\DD^{ch}_{\hbar,Y}) \cong \DD^{ch}_{\hbar}(Y) ~,
\end{equation}
as $\hbar$-adic vertex algebras. The global sections are also, therefore, equipped with an $\IG_\hbar$-equivariant embedding $\sigma_\hbar: \Gamma(\JJ T^*Y,\DD_{\hbar,Y}^{ch}) \hookrightarrow \DD^{ch}(Y)\Ph$ of $\hbar$-adic vertex algebras. This then implies the following lemma.

\begin{lem}\label{lem:Rees for h-adic cdos}
We have the following isomorphism of vertex algebras over $\ik$,
	$$\big(\Gamma(\JJ T^*Y, \DD_{\hbar,Y}^{ch})\tensor{\ik\Ph}\ik(\!(\hbar)\!)\big)^{\IG_\hbar}\cong \DD^{ch}(Y)= \Gamma(Y,\DD^{ch}_Y)$$
\end{lem}

\subsection{Chiral differential operators on a group}

Let $G$ be an algebraic group over $\IC$, then $T^*G$ is trivial and $G$ admits a sheaf of cdos. The construction of cdos over $G$ and their properties has been detailed in \cite{Gorbounov:2000cdosG,Arkhipov:2002cdosG}, here we present a brief review of their most salient properties.

The isomorphism classes of cdos on $G$ are in correspondence with a choice of a Chern--Simons cocycle in $H^3_{\rm dR}(G)$. Take $G$ to be $\GL_n$ or a simple Lie group, then there is a one parameter family of cdos, labeled by $k\in \IC$. For a fixed $k$, we denote the global sections $\Gamma(G,\DD^{ch}_k)$ by $\DD^{ch}_k(G)$. As a $\IN$-graded vertex algebra, $\DD^{ch}_k(G)$ is generated in degree $0$ by $\OO(G)$ and in degree $1$ by $\gf$, with singular OPEs
\begin{equation}
	\begin{split}
	x(z)a(w) &\sim \frac{x(a)(w)}{z-w}~,\\
	x(z)y(w) &\sim \frac{[x,y](w)}{z-w} + \frac{k(x,y)}{(z-w)^2}~,
	\end{split}
\end{equation}
for $x,y\in\gf$ and $a\in \OO(G)$.

From the trivialization of $TG$ as left invariant vector fields, we have an embedding of vertex algebras,
\begin{equation}
\Phi_L: V^k(\gf) \rightarrow \DD^{ch}_k(G)~.
\end{equation}
The trivialization of $TG$ as right invariant vector fields gives rise to a complementary embedding of vertex algebras
\begin{equation}
\Phi_R: V^{k-2h^\vee}(\gf) \rightarrow \DD^ch_k(G)~,
\end{equation}
but at the dual level. These are two mutually commuting current subalgebras and, moreover,
\begin{equation}
	{\rm Comm}_{\DD^{ch}_k(G)} \im\, \Phi_L = \im \, \Phi_R~, \text{ and } {\rm Comm}_{\DD^{ch}_k(G)} \im \, \Phi_R = \im \, \Phi_L~.
\end{equation}

From the construction of the previous section, we know that we can construct a microlocal version of $D^{ch}_{G,k}$ for any $k\in \IC$, call it $\DD^{ch}_{\hbar,G,k}$. The global sections $\Gamma(\JJ T^*G, \DD^{ch}_{\hbar,G,k}) \equiv \DD^{ch}_{\hbar,k}(G)$ are generated (over $\IC\Ph$) by $\OO(G)$ and $\gf$, with singular OPEs,

\begin{equation}
	\begin{split}
	x(z)a(w) &\sim \frac{\hbar x(a)(w)}{z-w}~,\\
	x(z)y(w) &\sim \frac{\hbar [x,y](w)}{z-w} + \frac{\hbar^2 k(x,y)}{(z-w)^2}~,
	\end{split}
\end{equation}
for $x,y\in\gf$ and $a\in \OO(G)$.

\bibliographystyle{amsalpha}
\bibliography{refs}

\end{document}